\documentclass[12pt]{amsart}
\usepackage{amscd,amsmath,amssymb,amsfonts,verbatim}
\usepackage[cmtip, all]{xy}
\usepackage{MnSymbol}
\usepackage{amsmath,mathtools}
\usepackage{mathtools}
\usepackage{xcolor}
\usepackage{color}
\usepackage{hyperref}
\usepackage{aliascnt}
\usepackage{tikz-cd}
\usepackage{graphicx} % for \scalebox
\usepackage{comment}

\usepackage{enumitem}
\setlist[enumerate]{label=$(\mathrm{\arabic*})$, leftmargin=*}
\setlist[itemize]{leftmargin=*}

\newtheorem{thm}{Theorem}[section]
\newaliascnt{theo}{thm}
\newtheorem{theo}[theo]{Theorem}
\aliascntresetthe{theo}

\newaliascnt{cor}{thm}
\newtheorem{cor}[cor]{Corollary}
\aliascntresetthe{cor}

\newaliascnt{defn}{thm}
\newtheorem{defn}[defn]{Definition}
\aliascntresetthe{defn}

\newaliascnt{prop}{thm}
\newtheorem{prop}[prop]{Proposition}
\aliascntresetthe{prop}

\newaliascnt{lem}{thm}
\newtheorem{lem}[lem]{Lemma}
\aliascntresetthe{lem}

\newaliascnt{conj}{thm}
\newtheorem{conj}[conj]{Conjecture}
\aliascntresetthe{conj}

\newaliascnt{que}{thm}

\aliascntresetthe{que}

\newaliascnt{ass}{thm}
\newtheorem{ass}[ass]{Assumption}
\aliascntresetthe{ass}

\newaliascnt{defnot}{thm}

\aliascntresetthe{defnot}

\newaliascnt{princ}{thm}

\aliascntresetthe{princ}

\newaliascnt{notn}{thm}
\newtheorem{notn}[notn]{Notation}
\aliascntresetthe{notn}

\newaliascnt{exmp}{thm}

\aliascntresetthe{exmp}

\theoremstyle{remark}
\newaliascnt{rem}{thm}
\newtheorem{rem}[rem]{Remark}
\aliascntresetthe{rem}

\newtheorem{exm}[thm]{Example}

\newcommand{\sO}{{\mathcal O}}

\newcommand{\A}{{\mathbb A}}

\newcommand{\G}{{\mathbb G}}

\newcommand{\Q}{{\mathbb Q}}

\newcommand{\Z}{{\mathbb Z}}

\newcommand{\Ker}{{\rm Ker}}

\newcommand{\alb}{{\rm alb}}
\newcommand{\CH}{{\rm CH}}

\newcommand{\inj}{\hookrightarrow}

\newcommand{\Pic}{{\rm Pic}}
\newcommand{\Div}{{\rm Div}}
\newcommand{\Ext}{{\rm Ext}}
\newcommand{\Hom}{{\rm Hom}}

\newcommand{\Spec}{{\rm Spec \,}}
\newcommand{\sing}{{\rm sing}}

\newcommand{\ab}{\rm ab}

\newcommand{\Gal}{{\rm Gal}}

\newcommand{\supp}{{\rm supp}\,}

\newcommand{\NS}{{\operatorname{NS}}}

\newcommand{\Zar}{{\text{\rm Zar}}}

\newcommand{\Sch}{{\operatorname{\mathbf{Sch}}}}

\newcommand{\Sm}{{\mathbf{Sm}}}

\newcommand{\Sym}{{\operatorname{\rm Sym}}}

\newcommand{\eff}{{\operatorname{\rm eff}}}

\newcommand{\Nis}{{\operatorname{Nis}}}

\newcommand{\ds}{{/\kern-3pt/}}

\newcommand{\res}{{\operatorname{res}}}

\newcommand{\cok}{{\operatorname{coker}}}

\newcommand{\sm}{{\operatorname{sm}}}

\newcommand{\Sus}{{\operatorname{sus}}}

\newcommand{\pr}{{\rm{pr}}}

\renewcommand{\dim}{\text{\rm dim}}

\newcommand{\tuborg}{\left\{\begin{array}{ll}}
	\newcommand{\sluttuborg}{\end{array}\right.}

\newcommand{\Kum}{{\rm Kum}}

\newcommand{\tor}{{\rm tor}}

\newcommand{\img}{{\rm Im}}

\newcommand{\dv}{{\rm div}}

\newcommand{\Aut}{{\rm Aut}}

\newcounter{elno}

\newcounter{elno-abc}   

\newcounter{elno-abc-prime}

\begin{document}
	\title{Zero-cycles on quasi-projective surfaces over $p$-adic fields}
		\author{Evangelia Gazaki and Jitendra Rathore}
		\address{Department of Mathematics, University of Virginia, 221 Kerchof Hall, 141 Cabell Dr., Charlottesville,
			VA, 22904, USA.}
			\email{eg4va@virginia.edu }
		\address{Department of Mathematics, South Hall, Room 6607, University of California Santa Barbara, CA, 93106-3080, USA.}
		\email{jitendra@ucsb.edu}

	\keywords{$p$-adic fields, 0-cycles, Milnor $K$-groups}

	\maketitle

\begin{abstract} A conjecture of Colliot-Th\'{e}l\`{e}ne predicts that for a smooth projective variety $X$ over a finite extension $k$ of  $\Q_p$ the kernel of the Albanese map $\CH_0(X)^{\deg=0}\to Alb_X(k)$ is the direct sum of a divisible group and a finite group. In this article we show that if $\pi:X\dashrightarrow Y$ is a generically finite rational map between smooth projective surfaces and the conjecture is true for $X\otimes_k L$ for every finite extension $L/k$, then it is true for $Y$. Using work of Raskind and Spiess, this proves the conjecture for surfaces that are geometrically dominated by products of curves, under some assumptions on the reduction type of the Jacobians. The method involves studying  similar questions for an open subvariety $U$ of a projective surface $X$ by replacing the Chow group of $0$-cycles  with Suslin's singular homology $H_0^\Sus(U)$. 
\end{abstract}

	\setcounter{tocdepth}{1}
	
	\tableofcontents

	\section{Introduction}

Let $X$ be a smooth projective variety over a $p$-adic field $k$. The Chow group of $0$-cycles, $\CH_{0}(X)$, is an important geometric invariant associated to $X$ and determining its structure for a general  $X$ remains a challenging task. This group admits a natural two-step filtration:
     \begin{center}
     	$0\subseteq F^{2}(X) \subseteq F^{1}(X) \subseteq \CH_{0}(X),$
     \end{center}
 where $F^{1}(X) :=  \ker(\CH_{0}(X) \xrightarrow{\deg} \mathbb{Z} )$ is the kernel of the degree map, and $F^{2}(X) :=  \ker(F^{1}(X) \xrightarrow{\alb_X} Alb_{X}(k) )$ is the kernel of the Albanese map of $X$. Here, $Alb_{X}$ is the Albanese variety assiociated to $X$. By a classical theorem of Mattuck \cite{Mat}, the structure of the group $Alb_{X}(k)$ is known. Therefore, the study of $F^{2}(X)$ becomes central.
 The following well-known conjecture of Colliot-Th\'{e}l\`{e}ne gives a prediction about the structure of the group $F^{2}(X)$ over $p$-adic fields.

 \begin{conj}\label{mainconj} (\cite[1.4 (d, e, f)]{CT2}, \cite[Conjecture 3.5.4]{RS})  Let $X$ be a smooth projective variety  over a finite extension $k$ of the $p$-adic field $\Q_p$. Then  
 	the kernel $F^{2}(X)$ of the Albanese map has a decomposition $F^{2}(X) \cong F \oplus D$, where $F$ is a finite group and $D$ is a divisible group. 
 	\end{conj}

A celebrated result in this direction was obtained by S. Saito and K. Sato (\cite[Theorem 0.3]{SS})\footnote{\cite{SS} assumes that $X$ has a semistable model over the ring of integers of $k$. See \cite[Th\'{e}or\`{e}me 0.2]{CT3} for the generalization to all smooth projective varieties.}, who proved that the degree $0$ subgroup $F^1(X)$ has a decomposition $F^1(X)\cong F\oplus D$, where $F$ is a finite group and $D$ is a group divisible by any integer coprime to the residue characteristic $p$. When $X$ is a curve, the Albanese map is an isomorphism, and hence the conjecture is trivially true. In higher dimensions the full conjecture is only known in the following very limited cases.

 \begin{enumerate}
 \item\label{item1} (\cite[Theorem 1.1]{RS}) For a product $X=C_1\times\cdots\times C_d$ of smooth projective curves with $X(k)\neq\emptyset$  and such that for all $i\in\{1,\ldots, d\}$ the Jacobian variety $J_i$ of $C_i$ has a mixture of  good ordinary and split multiplicative reduction (see \autoref{reddef}). When all factors are elliptic curves, this result has been extended to allow one curve to have good supersingular reduction (\cite[Theorem 1.2]{GL}). 
 \item (\cite[Theorem 1.2]{GL24}) For a $K3$ surface $X$ which is geometrically isomorphic to the Kummer surface $\Kum(A)$ of an abelian surface $A$ isogenous to a product of elliptic curves satisfying the reduction assumptions as above. 
 \end{enumerate}

\autoref{mainconj} can be restated as saying that the group $F^2(X)/F^2(X)_{\dv}$ is finite, where $F^2(X)_{\dv}$ is the maximal divisible subgroup of $F^2(X)$.
When $X$ is a surface, a result of Colliot-Th\'{e}l\`{e}ne about the $n$-torsion subgroup of $\CH_0(X)$ (\cite[Theorem 8.1]{CT1}) allows to reduce \autoref{mainconj} to showing the weaker property that the group $F^2(X)/F^2(X)_{\dv}$ is torsion of finite exponent. This in turn allows us to study surfaces with a certain geometric behavior; for example some of the assumptions for $C_1\times C_2$ can be weakened. 

  In this article we establish Conjecture \ref{mainconj} for a much larger collection of surfaces. Our first theorem is the following. 
 
 \begin{theo}\label{thm:main1intro} (see \autoref{thm:main1}) Let $X, Y$ be smooth projective surfaces over a $p$-adic field $k$. Suppose there is a generically finite rational map $\pi:X\dasharrow Y$. If \autoref{mainconj} is true for $X\otimes_k L$ for every finite extension $L/k$, then it is true for $Y$. The same result holds true if the map $\pi$ is defined over a finite extension of the base field. 
 \end{theo}

 Using the aforementioned result (1) of Raskind and Spiess, we can verify the conjecture for surfaces that are geometrically dominated by products of curves. 

 \begin{cor}\label{cor:newevidenceintro} (see \autoref{cor:newevidence})
     Let $X$ be a smooth projective surface over a $p$-adic field $k$ such that there is a finite rational map $\pi: C_1\times C_2\dashrightarrow X$ from a product of smooth projective curves whose Jacobians have a mixture of potentially good ordinary and multiplicative reduction. Then \autoref{mainconj} is true for $X$. 
 \end{cor}

There is a plethora of surfaces that are dominated by products of curves. Among those we highlight the following important classes. 
\begin{enumerate}
    \item[(a)]  Isotrivial fibrations $\pi:X\to C$. That is, fibrations for which all smooth fibers are isomorphic (see section \ref{sec:isotrivial}). 
    \item[(b)] Symmetric square $\Sym^2(C)$ of a genus $2$ curve and abelian surfaces (see section \ref{sec:Jac}).  
    \item[(c)] Fermat type diagonal surfaces $X_m=
    \{a_0x_0^m+a_1x_1^m+a_2x_2^m+a_3x_3^m\}\subset \mathbb{P}_k^3$, where $m\geq 4$ and $a_i\in k^{\times}$ for $i=0,\ldots, 3$ (see section \ref{sec:Fermat}).  
\end{enumerate}

Example (a) includes a very wide class of surfaces including isotrivial elliptic surfaces and many surfaces of general type. This also verifies the conjecture for many new classes of $K3$ surfaces. Namely, it is known that most $K3$ surfaces that admit a non-symplectic automorphism are isotrivial (\cite[Theorem 1.1]{GP}). Even though this class is special, it goes much beyond the geometrically Kummer surfaces considered in \cite{GL24}, and it even includes non elliptic $K3$'s (see \cite[Sections 7.2-7.4]{GP}). 
Example (b) together with the previous results on products of elliptic curves yield that \autoref{mainconj} is true for every abelian surface with a mixture of potentially good ordinary and multiplicative reduction (\autoref{cor:absurface}), which improves a weaker result obtained by the first author in \cite[Theorem 1.2]{Gaz19}. The case $m=4$ in Example (c) was already considered in \cite{GL24}. When $m\geq 5$, this is a surface of general type. Using known results about the Jacobian of the Fermat curve over finite fields (\cite{Yui}), we deduce that if $k$ is a finite extension of $\Q_p$ with $p\equiv 1\mod m$, then the surface $X_m$ satisfies \autoref{mainconj}.

More generally, if $G$ is a finite group acting faithfully on a product $C_1\times C_2$, then we obtain evidence for a surface $X$ which is a resolution of singularities of the quotient $(C_1\times C_2)/G$. Such surfaces are often called in literature \textit{product-quotient surfaces}, and over the complex numbers they have played a key role in the classification of surfaces with small birational invariants (see \cite{CP, MP, P}). 

In the special case of a Kummer surface considered in \cite{GL24} the proof of \autoref{mainconj} relied on a simple push-forward argument that reduced the question to the corresponding abelian surface. What made this work was that the singular surface $(E_1\times E_2)/\langle -1\rangle$ admits a very simple resolution of singularities. Unfortunately,  for most quotients $(C_1\times C_2)/G$, the resolution is more complicated and this argument is guaranteed not to work (see \autoref{rem:resolutionsings}).   The key new idea we introduce in this article is to work with open subvarieties that avoid singularities.

 \subsection*{Analogs for Quasi-projective Varieties}

Let $U$ be an open subvariety of a projective variety $X$ over a $p$-adic field $k$. An analog of $\CH_0(X)$ for $U$ is Suslin's singular homology group $H_0^{\Sus}(U)$ (see \autoref{sec:Suslin}). This group coincides with $\CH_0$ when $U$ is proper, but it is generally larger otherwise. Similarly to the proper case, there is a degree map $H_0^{\Sus}(U)\xrightarrow{\deg}\Z$. We will denote by $F^1(U)$ its kernel. Moreover, there is a \textit{generalized Albanese map}, \[F^1(U)\xrightarrow{\alb_U} Alb_U(k),\] where $Alb_U$ is a semi-abelian variety, which when $U(k)\neq\emptyset$ is universal for maps from $U$ to semi-abelian varieties. The variety $Alb_U$ was constructed by Serre in his \'{e}xpos\'{e} \cite{Se},  and it is an extension of $Alb_X$ by a torus.  We denote  $F^2(U):=\ker(\alb_U)$.  In this article we propose the following Conjecture, which is the analog of \autoref{mainconj} to the quasi-projective setting.

\begin{conj}\label{ourconj} (see \autoref{mainconj-quasiprojective}) Let $U$ be a smooth quasi-projective surface over a finite extension $k$ of the $p$-adic field $\Q_p$. 
Then the kernel $F^2(U)$ of the generalized Albanese map $F^1(U)\xrightarrow{\alb_U} Alb_U(k)$ admits a decomposition $F^2(U)\cong F\oplus D$, where $F$ is a finite group and $D$ is a divisible group.
\end{conj}

Our main evidence for this conjecture is the following theorem.

\begin{theo}\label{thm:main2intro} (see \autoref{thm:main2}, \autoref{maincor}) \autoref{mainconj} is true for smooth projective surfaces if and only if \autoref{ourconj} is true.
\end{theo} 
This theorem is the main technical result of the paper. \autoref{thm:main1intro} follows easily from this by a simple push-forward argument. 
What is more, \autoref{thm:main2intro} and its proof give essential information on the structure of Suslin's homology,  not only for surfaces that are dominated by products of curves, but also for those with geometric genus $p_g=0$.  When $X$ is smooth and the complement $X-U$ is $0$-dimensional, we show that the groups $F^2(X)$ and $F^2(U)$ coincide. On the other hand, when $X-U$ is $1$-dimensional, the group $F^2(U)$ can in general be much larger than $F^2(X)$. For example, when the divisor $D=X-U$ has more irreducible components than the order of the N\'{e}ron-Severi group $\NS(X)$, then the group $F^2(U)_{\dv}$ is in general nontrivial, even when $F^2(X)$ is finite. The Milnor $K$-group $K_2^M(k)\cong F^2(\G_m\times\G_m)$ is a classical example of this phenomenon. The same example also shows that the finite quotient $F^2(U)/F^2(U)_{\dv}$ can be much larger than $F^2(X)/F^2(X)_{\dv}$. 

An interesting case to consider is when the divisor $D$ has a single smooth irreducible component, in which case we have an equality $Alb_U=Alb_X$. In \autoref{sec:Suscomputs} we show that even  in this case the structure of $F^2(U)$ can vary very wildly.  The following proposition gives a situation where $F^2(U)=F^2(X)$.

\begin{prop} (\autoref{prop:secsoffibration}) Let  $\pi:X\to C$ be a smooth fibration, where $C$ is a smooth projective curve over a $p$-adic field $k$ with $C(k)\neq\emptyset$. Let $s:C\hookrightarrow X$ be a section of the fibration and  $U=X-s(C)$.  Then $F^2(U)=F^2(X)$. 
\end{prop}

 Using this we give nontrivial examples where $F^2(U)$ is finite (see Example \ref{ex:bielliptic}). 
At the other extreme, using the results of Colliot-Thélène (see \cite[Appendice B]{Sz}), we present various examples of curves \( C \subset X = \mathbb{P}^2 \) for which \( F^2(X - C)_{\mathrm{div}} \) is ``too" large (see Example~\ref{P2 example}). Additionally, based on results of M. Asakura and S. Saito (\cite{AS}), we provide an example of a ruled surface \( X \), with an open subscheme \( U \) defined as the complement of a divisor, for which the torsion subgroup \( F^2(U)_{\mathrm{tor}} \) is infinite even though \( F^2(X) = 0 \) (see Example~\ref{ruled-surface example}).

\subsection*{Outline}
 The proof of
\autoref{thm:main2} will occupy sections \S~2 - \S~4 %\eqref{sec:background}-\eqref{mainproofs} 
of this article. The method  involves three main steps. First, we use the interpretation of the groups $\CH_0(X)$ and $H_0^{\Sus}(U)$ as motivic cohomology groups (see \autoref{sec:motiviccoh}). The second key ingredient is the class field theory for curves over $p$-adic fields, as developed by S. Bloch (\cite{Bl1}) and S. Saito (\cite{Sai}); in particular the structure of the group $SK_1(C)$ for $C$ a smooth projective curve over a $p$-adic field. In \autoref{SK1(D)} we extend these results to $SK_1(D)$, where $D$ is a simple normal crossing divisor on a smooth projective surface $X$. Lastly, the most technical part of the proof is to analyze the toric piece of the generalized Albanese variety $Alb_U$ and compare it to the toric piece of $SK_1(D)$. The latter is generally of larger dimension, which we essentially annihilate by using some decomposable $(2,1)$-cycles on $X$. 

\autoref{thm:main1intro} follows easily from \autoref{thm:main2} using an argument similar to the one used in \cite[Theorem 1.2]{GL24}. This proof will be given in \autoref{sec:CTconj}. The rest of \autoref{sec:5} will be devoted to giving new examples of projective surfaces that satisfy \autoref{mainconj}, and to some computations of Suslin's homology.

\subsection*{Notation} In this article by a quasi-projective variety $U$ we will mean that $U$ admits an open embedding $j:U\hookrightarrow X$ to a projective variety $X$. For a variety $X$ over a field $k$ and a field extension $L/k$ we will denote by $X_L$ the base change to $L$. The function field of $X$ will be denoted $k(X)$. Moreover, for an integer $r\geq 0$ we will denote by $X_{(r)}$ the set of all schematic points on $X$ of dimension $r$. For $x\in X_{(r)}$ we denote $k(x)$ the residue field of $x$. For an abelian group $A$ we will denote by $A_{\dv}$ its maximal divisible subgroup  and by $A_{\tor}$ its torsion subgroup.  For an integer $n\geq 1$ we denote by $A_n$ and $A/n$ the kernel and cokernel of the multiplication by $n$ map $A\xrightarrow{n}A$ respectively. 

\subsection*{Acknowledgement} The first author's research was partially supported by the  NSF grant DMS-2302196.  The second author would like to thank the University
of Virginia for hosting his visit, during which substantial work was done. We are also grateful to Professors Jean-Louis Colliot-Th\'{e}l\`{e}ne, Jonathan Love, Shuji Saito and Takao Yamazaki  for many helpful discussions. 
	
\vspace{3pt}
	
\section{Background}\label{sec:background}
 In this section we recall some necessary background on Suslin's homology, the generalized Albanese map, motivic cohomology of schemes, and the class field theory for curves over $p$-adic fields.  

\subsection{Generalized Albanese Variety}\label{genalbsection} For a  projective variety $X$ over a field $k$ such that $X(k)\neq\emptyset$, the usual Albanese variety $Alb_X$ is an abelian variety such that for every $x_0\in X(k)$ there is a morphism $\phi_{x_0}:X\to Alb_X$ sending $x_0$ to the zero element of $Alb_X$, and if $f:X\to B$ is any morphism to an abelian variety $B$ with $f(x_0)=0_B$, then $f$ factors uniquely through a homomorphism $Alb_X\to B$. We recall that $Alb_X$ is the dual abelian variety to the Picard variety $\Pic^0_X$ of $X$. 
	
Serre in his \'{e}xpos\'{e} (\cite{Se}) constructed an analog for an open subvariety $U$ of a projective variety $X$. In what follows we assume that $U(k)\neq\emptyset$, and we fix a basepoint $x_0\in U(k)$. The Albanese variety $Alb_U$ is a semi-abelian variety over $k$ equipped with a morphism $\psi_{x_0}:U\to Alb_U$ which is universal for morphisms from $U$ to semi-abelian varieties sending $x_0$ to zero.
When the complement $X-U$ is of codimension at least $2$,  $Alb_U=Alb_X$. 
Next suppose that $X-U$ is the support of a divisor $D$ on $X$. In this case $Alb_U$ is an extension of $Alb_X$ by a torus $T$. Let us recall how this extension is obtained. Write $D=D_1+\cdots+D_r$, where each $D_i$ is a prime divisor of $X$. We assume for simplicity that each connected component $D_i$ of $D$ is geometrically connected. Let 
$\displaystyle I=\Z^r=\bigoplus_{i=1}^r\Z((D_i)_{\overline{k}})$ be the subgroup of $\Div(X_{\overline{k}})$ consisting of divisors supported in $D_{\overline{k}}$. Since $D$ is defined over $k$, $I$ is stable under the $G_k:=\Gal(\overline{k}/k)$-action on $\Div(X_{\overline{k}})$. In fact, the assumption that each $D_i$ is geometrically connected implies that $I$ has trivial $G_k$-action, and hence $I\simeq\Z^r$ as a $G_k$-module.  
Let $J$ be the following kernel, 
\[J=:\ker[I\hookrightarrow \Div(X_{\overline{k}})\to \Pic(X_{\overline{k}})\to \NS(X_{\overline{k}})],\]
and $j$ be the inclusion $j:J\hookrightarrow I$. Notice that $J$ comes equipped with a natural homomorphism $\varphi:J\to\Pic^0(X_{\overline{k}})$. Since $I$ is a subgroup of $\Div(X_{\overline{k}})^{G_k}=\Div(X)$, and we have an equality $\Pic^0(X_{\overline{k}})^{G_k}=\Pic^0(X)$ (given $X(k)\neq\emptyset$), 
we can descend $\varphi$ to a homomorphism \[\varphi:J\to\Pic^0(X)=\Pic^0_X(k).\] 
Since $J\simeq\Z^{r-n}$ for some $n\geq 1$, $\varphi$ corresponds to a tuple $\varphi=(\varphi_1,\ldots, \varphi_{r-n})\in\Pic^0_X(k)^{r-n}$. 
Recall that there is an isomorphism $\Ext^1(Alb_X, \G_m)\simeq\Pic_X^0(k)$, where $\Ext^1$ is taken in the category of algebraic groups over $k$ (see \cite[I.3, Lemma 3.1]{Mil06}). It then follows that $\varphi$ gives rise to an element $Alb_U\in\Ext^1(Alb_X,\G_m^{r-n})$. Set $T=\G_m^{r-n}$. Then $Alb_U$ fits into a short exact sequence 
\[0\to T\to Alb_U\to Alb_X\to 0.\]
Moreover, the free abelian group $J=\Z^{r-n}$ is precisely the character group $X^\star(T)=\Hom(T,\G_m)$ of the torus $T$. 
\begin{rem}
Without the assumption that each irreducible component $D_i$ of $D$ is geometrically connected, the resulting torus $T$  in general is not split. In what follows we will be making the even stronger assumption that each irreducible component $D_i$ has a $k$-rational point.     
\end{rem}

\begin{exm}\label{onecomponent} 
    Let $C\hookrightarrow X$ be a smooth projective connected curve and let $U=X-C$. Since $[C]$ is a $\Z$-linearly independent element of $\NS(X)$, it follows that $r=n=1$, and hence $Alb_X=Alb_U$.  This observation will be used in \autoref{sec:Suscomputs}. 
\end{exm}

\subsection{Suslin's Homology}\label{sec:Suslin} Let $U$ be a smooth quasi-projective variety over a field $k$. In \cite{Sus} Suslin and Voevodsky defined a class group $H_0^{\Sus}(U)$ of zero-cycles on $U$, which  coincides with the Chow group $\CH_0(U)$ when $U$ is proper, but it is larger otherwise. We recall the definition below. 
\begin{defn} (\cite{Sus})
Let $\displaystyle Z_0(U)=\bigoplus_{x\in U_{(0)}}\Z(x)$ be the free abelian group on all closed points of $U$. Then $H_0^{\Sus}(U)$ is the quotient of $Z_0(U)$ by the subgroup $R$ generated by $i_0^\star(Z)-i_1^\star(Z)$, where $i_\nu:U\to U\times \A^1$ stands for the inclusion $x\mapsto (x,\nu)$ for $\nu=0,1$, and $Z$ runs through all closed integral subschemes of $U\times\A^1$ such that the projection $Z\to\A^1$ is finite and surjective. 
\end{defn} This group is often referred to in the literature as \textit{Suslin's singular homology}. 

\subsection*{Relation to Wiesend's Class group}
When the base field $k$ is perfect, the group $H_0^\Sus(U)$ coincides with another class group, $C_0(U)$, called \textit{Wiesend's $0$-th ideal class group}. The name reflects that the definition was inspired by a similar group introduced by G. Wiesend (\cite{Wie}) for arithmetic schemes (also see \cite{Yam1}). 
We recall the set-up below. 
\begin{notn}
    Let $y\in U_{(1)}$. We will denote by $C(y)$ the closure of the point $y$ in $U$ (which is a closed integral curve in $U$), $\widetilde{C}(y)\twoheadrightarrow C(y)$ its normalization, $\widetilde{C}(y)\hookrightarrow\overline{C}(y)$ its smooth completion, and $C_\infty(y)=\overline{C}(y)-\widetilde{C}(y)$. 
\end{notn}

\begin{defn} (cf.~\cite[Definition 1.1]{Yam1})
Wiesend's $0$-th ideal class group is defined to be the quotient of $Z_0(U)$ by the subgroup $R'$ generated by divisors of functions $\dv(f)$, where $f\in k(y)^\times$ for some $y\in U_{(1)}$ satisfying $f=1$ on $C_\infty(y)$. 
\end{defn} 

It is a theorem of A. Schmidt (\cite[Theorem 5.1]{Sch}) that when $k$ is perfect there is an isomorphism 
\[H_0^\Sus(U)\cong C_0(U).\] In what follows we will be mostly using Suslin's homology, but we won't be distinguishing between these two groups. 

The group $H_0^\Sus(U)$ enjoys similar properties as the Chow group of zero-cycles $\CH_0(X)$ for $X$ projective. Namely, there is a well-defined degree map,
\[\deg:H_0^\Sus(U)\to\Z,\;\; [x]\mapsto [k(x):k].\] 
We will denote by $F^1(U)$ the subgroup of degree $0$ elements. That is, $F^1(U):=\ker(\deg)$. 
From now on assume that $U(k)\neq\emptyset$. Then the Albanese map $U\to Alb_U$ considered in \autoref{genalbsection} induces a homomorphism
$$\alb_U: F^1(U)\to Alb_U(k)$$  
that does not depend on the choice of base point $x_0\in U(k)$. 
A proof of why $\alb_U$ is well-defined using Suslin's definition can be found in \cite[Section 3]{SpSz}. 

\begin{defn} Suppose $U(k)\neq\emptyset$. 
    We define $F^2(U):=\ker[F^1(U)\xrightarrow{\alb_U}Alb_U(k)]$ to be the kernel of the generalized Albanese map. We thus have a filtration,
    \[H_0^\Sus(U)\supseteq F^1(U)\supseteq F^2(U)\supseteq 0.\] 
\end{defn}

\subsubsection{Relation to Generalized Jacobians}\label{sec:genJac}
    Wiesend's definition is often easier to use in applications. For example, it follows directly that when $U=\overline{C}-C_\infty$, where $\overline{C}$ is a smooth projective curve over a perfect field $k$ and $C_\infty$ is a reduced closed subvariety, then $C_0(U)$ is the group of classes of divisors on $\overline{C}$ prime to $C_\infty$ modulo $C_\infty$-equivalence defined by Serre in \cite[Chapter V.2]{Se2}. Its degree $0$ subgroup, $F^1(U)$, coincides with the generalized Jacobian $J_\mathfrak{m}$ of $\overline{C}$ corresponding to the modulus $\displaystyle \mathfrak{m}=\sum_{x\in C_\infty}x$. 
    Using this observation, one can immediately prove that for an arbitrary smooth quasi-projective variety $U$ the Albanese map $\alb_U: F^1(U)\to Alb_U(k)$ is well-defined using the universal property of generalized Jacobians.

	\subsection{Motivic Cohomology}\label{sec:motiviccoh}

    In this subsection, we recall some results from motivic cohomology for schemes over a field. For a detailed exposition, we refer the reader to \cite{Voe}, \cite{MVW}, \cite{CD2}, \cite{CD1} (see also \cite[Section~5]{GKR}). Here, we recall them briefly, restricting to what is relevant for this manuscript. Throughout this subsection, we denote $k$ to  be any field of characteristic  $0$.  For a Noetherian scheme $X$, we shall denote the category of separated schemes which are of finite type over $X$ by $\Sch_{X}$. We shall let $\Sm_{X}$ denote the category of smooth schemes over $X$. 

In \cite{Voe} (also \cite{MVW}), Voevodsky defined the rigid triangulated category of effective motivic complexes, denoted as $\textbf{DM}^{\eff,-}_{\Nis} (k)$. For a given Noetherian $k$-scheme $X$, there is a monoidal triangulated category of mixed motives  
$\textbf{DM}(X, \mathbb{Z})$ \cite[Theorem 11.4.5]{CD1} which agrees with the construction of Voevoedsky \cite{Voe} when $X = \Spec(k)$. We denote $ \underline{\textbf{DM}}_{cdh}(X, \mathbb{Z})$, the category defined in similar manner as $\textbf{DM}(X,  \mathbb{Z})$, where  one replaces the Nisnevich topology on $\Sm_{X}$ by the cdh-topology on $\Sch_{X}$. We let $\textbf{DM}_{cdh}(X,\mathbb{Z})$ be the full localizing triangulated subcategory of $ \underline{\textbf{DM}}_{cdh}(X, \mathbb{Z})$ generated by motives of the form $M_{X}(Y)(n)$ for $Y \in \Sm_{X}$ and $n \in \mathbb{Z}$. The assignment $X \mapsto \textbf{DM}_{cdh}(X, \mathbb{Z})$ satisfies Grothendieck's six functor formalism (\cite[A.5]{CD1}, \cite{CD2}). We denote $\mathbb{Z}_{X}$, the constant Nisnevich (resp. cdh) sheaf with transfers, which is also the identity object for the monoidal structure of $\textbf{DM}(X, \mathbb{Z})$ (resp. $\textbf{DM}_{cdh}(X, \mathbb{Z})$). We use the notation $\mathbb{Z}_{k}$ to denote $\mathbb{Z}_{\Spec(k)}$.

  Let $X$ be a finite type Noetherian scheme over $k$ with $f : X \rightarrow \Spec(k)$, the structure map. We then have the following objects
\begin{center}
	$M_{k}(X) = f_{\sharp}\mathbb{Z}_{X} \hspace{1mm} {\cong}^{1} \hspace{1mm} f_{!}f^{!}\mathbb{Z}_{k},  \hspace{1mm}  $and$  \hspace{2mm} M_{k}^{c}(X) = f_{\ast}f^{!}\mathbb{Z}_{k}$;
\end{center} 
in $\textbf{DM}(\Spec(k), \mathbb{Z}) \cong \textbf{DM}_{cdh}(\Spec(k), \mathbb{Z})$ (\cite[Corollary 5.9]{CD2}). All the above functors (except $f_{\sharp}$) are part of six functor formalism and the functor $f_{\sharp}$ is an adjoint to the functor $f^{\ast}$ (see \cite[(Remark 8.5)]{CD2}). Also, the isomorphism ${\cong}^{1}$ follows from \cite[(8.7.1)]{CD2}. 

\begin{defn}
For integers $q$ and $i$, we define the (motivic) cohomology, homology, cohomolgy with compact support as follows:

\medskip

\begin{center}
	$H_{M}^{i}(X, \mathbb{Z}(q)) = \Hom_{\mathbf{DM}_{cdh}(\Spec(k), \mathbb{Z})} (M_{k}(X), \mathbb{Z}_{k}(q)[i] )$;
	
	\vspace{1mm}
	
	$H^{M}_{i}(X, \mathbb{Z}(q)) = \Hom_{\mathbf{DM}_{cdh}(\Spec(k), \mathbb{Z})} (\mathbb{Z}_{k}(q)[i], M_{k}(X))$;
	
	\vspace{1mm}
	
		$H_{M,c}^{i}(X, \mathbb{Z}(q)) = \Hom_{\mathbf{DM}_{cdh}(\Spec(k), \mathbb{Z})} (M_{k}^{c}(X), \mathbb{Z}_{k}(q)[i] )$.
	
\end{center}

\end{defn}

\vspace{1mm}

\vspace{1mm}

It follows by \cite[Corollary 5.9]{CD2} that the change of topology functor $\textbf{DM}(X, \mathbb{Z}) \rightarrow  \textbf{DM}_{cdh}(X, \mathbb{Z})$ is an equivalence of monoidal triangulated categories if $X \in \Sm_{k}$. This implies that for $X \in \Sm_{k}$, the groups $H_{M}^{i}(X, \mathbb{Z}(q)) $ and $H^{M}_{i}(X, \mathbb{Z}(q)) $ can also be expressed as:

\begin{center}
	$H_{M}^{i}(X, \mathbb{Z}(q)) \cong \Hom_{\textbf{DM}^{\eff,-}_{\Nis} (k)} (M_{k}(X), \mathbb{Z}_{k}(q)[i] )$;
	
	\vspace{1mm}
	
	$H^{M}_{i}(X, \mathbb{Z}(q)) \cong \Hom_{\textbf{DM}^{\eff,-}_{\Nis} (k)} (\mathbb{Z}_{k}(q)[i], M_{k}(X))$.
	
	\vspace{1mm}

\end{center}
Moreover  in this case for any integer $q \geq 0$, there is a complex $\mathbb{Z}(q)_{X}$ of Zariski sheaves $\mathbb{Z}(q)_{X}$ on $X_{\Zar}$ (see \cite[Definition 3.1]{MVW}) such that the motivic cohomology groups for $X$ are isomorphic to the hypercohomology of the complex $\mathbb{Z}(q)_{X}$. In other words, for $X \in \Sm_{k}$, we have 
\begin{center}
	$H_{M}^{i}(X, \mathbb{Z}(q)) \cong \mathbb{H}_{\Zar}^{i} (X, \mathbb{Z}(q))$, and
\end{center}
\begin{equation}\label{motivic-higher-chow-iso}
    H_{M}^{i}(X, \mathbb{Z}(q)) \cong \CH^{q}(X, 2q-i)
\end{equation} (see \cite[Theorem 19.1]{MVW}, \cite[Corollary 8.12]{CD2}), where $\CH^{q}(X,i)$ denotes the Bloch's higher Chow groups (\cite{Bl2}, \cite[Definition 17.1]{MVW}).

\begin{theo}\label{thm:compactly-support-suslin}
    Let $X$ be a smooth scheme over $k$ of pure dimension $d$. %and $j : U \hookrightarrow X$ be an open immersion. 
    Then we have an isomorphism,
    \begin{equation}\label{compact-motivic-Sus-iso}
        H^{2d}_{M,c}(X, \mathbb{Z}(d)) \cong H^{\Sus}_{0}(X).
    \end{equation}
\end{theo}
\begin{proof}
 By the duality isomorphism, we have
 $ H_{M,c}^{2d}(X, \mathbb{Z}(d)) \cong H^{M}_{0}(X, \mathbb{Z}(d))$ (\cite[Remark 2.2]{Yam1}). Also, we have $H^{M}_{0}(X, \mathbb{Z}(d)) \cong H^{\Sus}_{0}(X)$ (\cite{Sus}, \cite[Theorem 1.3]{Yam1}). Combining the above two isomorphisms, we get the desired result. 
 
 \end{proof}

\begin{lem}\label{lem:cohomology Ex-seq}
	Let $X $ be a projective variety over $k$,  $\iota \colon W \inj X$ be a closed immersion, and 
	$u \colon U \inj X$ the inclusion of the complement. There is a long exact sequence
	
	\begin{center}
		$ \hspace*{0.5cm}   \cdots \to H^i_{M,c}(U, \mathbb{Z}(j)) \xrightarrow{u_*}
		H^i_M(X, \mathbb{Z}(j)) 
		\xrightarrow{\iota^*}  H^i_{M}(W, \mathbb{Z}(j)) \xrightarrow{\partial}
		H^{i+1}_{M,c}(U, \mathbb{Z}(j)) \to \cdots $
	\end{center}

\end{lem}
\begin{proof}
By \cite[Theorem 5.11]{CD2} (also \cite[Lemma 8.4]{GKR}) we have 
\begin{center}
		$  \hspace*{0.5cm}   \cdots \to H^i_{M,c}(U, \mathbb{Z}(j)) \xrightarrow{u_*}
		H^i_{M,c}(X, \mathbb{Z}(j)) 
		\xrightarrow{\iota^*}  H^i_{M,c}(W, \mathbb{Z}(j)) \xrightarrow{\partial}
		H^{i+1}_M,c(U, \mathbb{Z}(j)) \to \cdots $
	\end{center}
As $W$ and $X$ are projective,  it follows that $H^i_{M,c}(W, \mathbb{Z}(j)) \cong H^i_{M}(W, \mathbb{Z}(j))$ and $H^i_{M,c}(X, \mathbb{Z}(j)) \cong H_{M}^i(X, \mathbb{Z}(j))$, which concludes the proof of the lemma.

\end{proof}

\subsection{$K$-theory for Curves over $p$-adic Fields}\label{k-theory-curves}\label{sec:SK1}
In this subsection we recall some necessary facts from the Class field theory for curves over $p$-adic fields as developed by S. Bloch (\cite{Bl1}) and S. Saito (\cite{Sai}). 

Let $C$ be a smooth projective curve over a field $k$. We consider the $K$-group
\[SK_1(C):=\cok\left[\bigoplus_{x\in C_{(0)}}\partial_x:K_2(k(C))\to \bigoplus_{x\in C_{(0)}}k(x)^\times\right],\]    where for each closed point $x\in C$, $\partial_x: K_2(k(C))\to k(x)^\times$ is the boundary map in $K$-theory.  The pointwise norm maps $N_{k(x)/k}:k(x)^\times\to k^\times$
induce a \textit{norm map} 
\[N:SK_1(C)\to k^\times,\] 
whose kernel is traditionally denoted by $V(C)$,  %. The norm map is surjective, 
yielding an exact sequence 
\[0\to V(C)\to SK_1(C)\xrightarrow{N} k^\times.\] 
Moreover, when $C$ has a $k$-rational point $x_0\in C(k)$, the closed immersion $\iota_{x_0}:\Spec(k)\hookrightarrow C$ gives a splitting to the norm map, and hence we have a decomposition 
\begin{equation}\label{eq0}
    SK_1(C)\cong V(C)\oplus k^\times.
\end{equation} 

From now on, we assume that $k$ is a finite extension of the $p$-adic field $\Q_p$.  In this case, Bloch (\cite{Bl1}) and Saito (\cite{Sai})  constructed \textit{reciprocity maps}  
\[\sigma: SK_1(C)\to\pi_1^{\ab}(C) \;\;\;\;\;\text{and}\;\;\;\;\;\tau: V(C)\to T_{G_k},\]
where $\pi_1^{\ab}(C)$ is the abelian \'{e}tale fundamental group of $C$, $T$ is the Tate module associated to the Jacobian $J_C$ of $C$, $G_k=\Gal(\overline{k}/k)$, and $T_{G_k}$ are the $G_k$-coinvariants.  Bloch and Saito studied the kernel and the image of these maps when the Jacobian $J_C$ has respectively good and bad reduction. The following theorem summarizes some of the main results in the class field theory of smooth projective curves. 
\begin{theo}\label{CFTcurves} (\cite[Theorem 2.9]{Bl1}, \cite[Chap. II, Theorems 5.1 \& 4.1]{Sai}). Let $C$ be a smooth projective geometrically connected curve over a $p$-adic field $k$. Then the following are true:
\begin{enumerate}
    \item[(i)] The kernel of $\sigma$ (resp. $\tau$) is the maximal divisible subgroup of $SK_1(C)$ (resp. $V(C)$). 
    \item[(ii)] The image of $\tau$ is finite. 
\end{enumerate}
\end{theo}

Note that since $k$ is $p$-adic, the multiplicative group $k^\times$ has no nonzero divisible elements, and hence the maximal divisible subgroups of $SK_1(C)$ and $V(C)$ coincide. Thus, the above theorem yields the following Corollary.

\begin{cor}\label{cor:SK1structure}
    Let $C$ be a smooth projective connected curve over a $p$-adic field $k$ with $C(k)\neq\emptyset$. Then we have a decomposition 
    \[SK_1(C)\simeq SK_1(C)_{\dv}\oplus F\oplus k^\times,\] where $F$ is some finite group. 
\end{cor}

\vspace{3pt}	

\section{Preparatory Results}\label{sec:prelims}
In this section we prove a number of preliminary results that will be used to prove \autoref{thm:main2intro} in the next section.

\subsection{Torsion-by-divisible}
Let $A$ be an abelian group and $A_{\dv}$ its maximal divisible subgroup. Then the short exact sequence 
\[0\to A_{\dv}\to A\to A/A_{\dv}\to 0\] splits, since divisible groups are injective $\Z$-modules. If the group $A/A_{\dv}$ is torsion of finite exponent, we will say that $A$ is \textit{torsion of finite exponent by divisible}. 

The following lemma is probably well-known to the experts.   

\begin{lem}\label{lem:F+D} Let $A\xrightarrow{f} B\xrightarrow{g} C\to 0$ be an exact sequence of abelian groups. If the groups $A,C$ are torsion of finite exponent by divisible, then the same is true for $B$.    
\end{lem}	
\begin{proof}

Consider the short exact sequence $ 0 \rightarrow A'\xrightarrow{i} B\xrightarrow{g} C\to 0$, where $A'$ is the image of $f$. Since $A \cong A_{\dv} \oplus A/A_{\dv}$, where $A/A_{\dv}$ is torsion of finite exponent, we have $A' = f(A_{\dv}) + f(A/A_{\dv})$ with $D=f(A_{\dv})$  divisible and $T =  f(A/A_{\dv})$ torsion of finite exponent. Therefore,  $\displaystyle A' = D + T  \cong D \oplus \frac{T}{T \cap D}$.  So, without loss of generality, one can assume that $f$ is injective. 

Note that for any abelian group $G$, we have \[G_{\dv} = \img\left[\operatorname{Hom}(\mathbb{Q}, G) \xrightarrow{i_{G}^{\ast}} \operatorname{Hom}(\mathbb{Z}, G) = G\right],\] where the later map is induced from the inclusion $\mathbb{Z} \hookrightarrow \mathbb{Q}$. Since $A$ is torsion of finite exponent by divisible, it follows that the group $\operatorname{Ext^{1}_{\mathbb{Z}}}(\mathbb{Q}, A)  $ vanishes (see \cite[Proposition 7.33]{Rot}). 
To finish the proof, we now consider the following commutative diagram with exact rows:
\[
\begin{tikzcd}
0 \arrow[r] &  \operatorname{Hom}(\mathbb{Q}, A) \arrow[r, "f^\star"] \arrow[d, "i_{A}^{\ast}"] & \operatorname{Hom}(\mathbb{Q}, B) \arrow[r, "g^\star"] \arrow[d, "i_{B}^{\ast}"] & \operatorname{Hom}(\mathbb{Q}, C) \arrow[r] \arrow[d, "i_{C}^{\ast}"] & 0 \\
0 \arrow[r] & \operatorname{Hom}(\mathbb{Z}, A) \arrow[r] & \operatorname{Hom}(\mathbb{Z}, B)\arrow[r] & \operatorname{Hom}(\mathbb{Z}, C) \arrow[r] & 0.
\end{tikzcd} 
\]
Note that the top sequence is exact since $\operatorname{Ext^{1}_{\mathbb{Z}}}(\mathbb{Q}, A)$ vanishes, and the bottom short exact sequence is the given one. Since $A/A_{\dv}$ and $C/C_{\dv}$ are  both torsion of finite exponent, the Snake lemma yields that $B/B_{\dv}$ is torsion of finite exponent, which completes the proof of the lemma.

	\end{proof}

\subsection{N\'{e}ron-Severi computations}\label{NSXstuff}

In this subsection we assume that $X$ is a smooth projective surface over a perfect field $k$ and $D=D_1+\cdots+D_r$ a simple normal crossing divisor on $X$. We assume that the geometric N\'{e}ron-Severi group $\NS(X_{\overline{k}})$ has trivial $\Gal(\overline{k}/k)$-action, and that $D_i(k)\neq\emptyset$ for $i=1,\ldots, r$. In particular, each $D_i$ is geometrically connected. 
Since each $D_i$ is a smooth projective curve over $k$ with a $k$-rational point, $\NS(D_i)\simeq\Z$ via the degree map. Let $\iota:\supp(D)\hookrightarrow X$ be the closed embedding, which induces a pullback map 
	\[\iota^\star:\NS(X)\to\bigoplus_{i=1}^r \NS(D_i)=\Z^r.\] 
    This sends the class $[Z]$ of a divisor $Z$ in $X$ to $(Z\cdot D_1,\ldots, Z\cdot D_r)$, where $\cdot$ is the intersection product. 
For $i=1,\ldots, r$ let $[D_i]$ be the class of $D_i$ in $\NS(X)$. Let 
$s=r-n$ be the dimension of the maximal subtorus of $Alb_U$ as in \autoref{genalbsection}. Recall that $n$ is precisely the number of a maximal $\Z$-linearly independent subset of $\{[D_1],\ldots, [D_r]\}$ when considered as a subset of $\NS(X)$. Without loss of generality, we may assume that the elements $[D_1],\ldots, [D_n]$ are $\Z$-linearly independent.
	The following proposition is a key ingredient in the proof of  \autoref{thm:main2intro}. 
	
\begin{prop}\label{prop:NSXcomputations} Consider the set-up of the previous paragraph. Then $\iota^\star(\NS(X))$ is a subgroup of $\Z^r$ of rank $n$. 
\end{prop}
\begin{proof} 
By assumption the elements $[D_1],\ldots, [D_n]$ are $\Z-$linearly independent in $\NS(X)$. For each $m=1,\ldots, n$ we denote by $\pr_m:\Z^r\to\Z^m$ the projection onto the first $m$ factors. We consider the composition 
\[f_m:\NS(X)\xrightarrow{\iota^\star}\Z^r\xrightarrow{\pr_m}\Z^m.\]
We claim that this map has finite cokernel for every $1\leq m\leq n$. 
We proceed by induction on $m\geq 1$. Since $[D_1]\in \NS(X)$ is $\Z$-linearly independent, it is not numerically equivalent to $0$, and hence there is some curve $C$ in $X$ such that $C\cdot D_1\neq 0$. This proves the $m=1$ case. 

Next suppose $2\leq m\leq n$. By induction hypothesis there are divisors $S_1,\ldots, S_{m-1}$ in $X$ such that $\{\pr_{m-1}\circ\iota^\star([S_i]):1\leq i\leq m-1\}$ is a $\Z$-linearly independent subset of $\Z^{m-1}$. For each $i, j\in\{1,\ldots, m-1\}$ set $\mu_{ij}:=S_i\cdot D_j$. It follows that the matrix $A=(\mu_{ij})\in M_{m-1}(\Q)$ is invertible. It is clear that the set $\{f_{m}([S_1]),\ldots, f_{m}([S_{m-1}])\}\subset\Z^{m}$ is also $\Z$-linearly independent as the first $m-1$ coordinates of these vectors are the same as before. This means that $f_m(\NS(X))$ is a subgroup of $\Z^m$ of rank at least $m-1$. Suppose for contradiction that $\pr_m\circ\iota^\star$ does not have finite cokernel. This means that $\{f_{m}([S_1]),\ldots, f_{m}([S_{m-1}])\}\subset\Z^{m}$ is a maximal $\Z$-linearly independent subset of $f_m(\NS(X))$. It then follows that for every curve $C$ in $X$ there exist $q_{1,C},\ldots, q_{m-1,C}\in\Q$ such that 
\begin{equation}\label{eq10}
    (C\cdot D_1,\ldots, C\cdot D_m)=\sum_{i=1}^{m-1} q_{i,C}(S_i\cdot D_1,\ldots, S_i\cdot D_m).
\end{equation}
Let us rewrite the first $m-1$ coordinates of (\ref{eq10}) as follows: 
\[
\left\{\begin{array}{c}
    C\cdot D_1=\sum_{i=1}^{m-1} q_{i,C}\mu_{i1}   \\
     \cdots\\
     C\cdot D_{m-1}=\sum_{i=1}^{m-1} q_{i,C}\mu_{i,m-1}
\end{array}\right.
\] This can be thought of as a $\Q$-linear system with $m-1$ equations in the unknowns $q_{1,C},\ldots, q_{m-1,C}$. The matrix of this system is the transpose of $A$. Since this matrix is invertible, it follows that the system has a unique solution. In particular, for each $i=1,\ldots, m-1$, $q_{i,C}$ has an expression 
\begin{equation}
    q_{i,C}=\sum_{j=1}^{m-1}\lambda_{ij} (C\cdot D_j),
\end{equation} for some $\lambda_{ij}\in\Q$, which are independent of $C$. 

For $i\in\{1,\ldots, m-1\}$ set $\gamma_i=S_i\cdot D_m$, which is again independent of $C$. 
Combining all the above, the last coordinate of (\ref{eq10}) can be rewritten as 
\begin{eqnarray*}
 C\cdot D_m=&&\sum_{i=1}^{m-1} q_{i,C}(S_i\cdot D_m)=\sum_{i=1}^{m-1} q_{i,C}\gamma_i\\
 =&&\sum_{i=1}^{m-1}\gamma_i\left(\sum_{j=1}^{m-1}\lambda_{ij} (C\cdot D_j)\right).
\end{eqnarray*}
If we expand this relation we potentially reach an expression of the form
\[C\cdot\left(D_m-\sum_{i=1}^{m-1}\alpha_i D_i\right)=0,\] for some $\alpha_i\in\Q$ independent of $C$. Write $\alpha_i=\frac{a_i}{b}$ for some $a_i, b\in\Z$ with $b\neq 0$. Since the curve $C$ was arbitrary, it follows that the divisor $\displaystyle bD_m-\sum_{i=1}^{m-1}a_i D_i$ is numerically equivalent to $0$. But this contradicts the $\Z$-linear independence of $[D_1],\ldots, [D_m]$ for $m\leq n$ in $\NS(X)$. We conclude that $f_m(\NS(X))$ is a subgroup of $\Z^m$ of rank $m$, which completes the induction.

To finish the proof we need to show that the rank of $\iota^\star(\NS(X))$ is exactly $n$. For simplicity of notation we will only prove that $f_{n+1}(\NS(X))\otimes\Q$ is an $n$-dimensional subspace of $\Q^{r}$. The argument for a general $n+1\leq m\leq r$ is similar. By assumption, the subset $\{[D_1],\ldots, [D_n]\}\subset\NS(X)$ is $\Z$-linearly independent, while $\{[D_1],\ldots, [D_n], [D_{n+1}]\}$ is not. Working in $\NS(X)\otimes\Q$, this means that we can find $\lambda_1,\ldots, \lambda_n\in\Q$ such that $\displaystyle D_{n+1}-\sum_{i=1}^n \lambda_i D_i$ is numerically equivalent to $0$. Let $S\in\Div(X)$.  Then we can write
\begin{eqnarray*}
    f_{n+1}(S)=&&(S\cdot D_1,\ldots, S\cdot D_n, S\cdot D_{n+1})=\left(S\cdot D_1,\ldots, S\cdot D_n, \sum_{i=1}^n \lambda_i (S\cdot D_{i})\right)\\
    =&& S\cdot D_1(1,0,\ldots, 0, \lambda_1)+\cdots + S\cdot D_n (0,\ldots, 1, \lambda_n)\\
    &&=c_1(1,0,\ldots, 0, \lambda_1)+\cdots+c_n(0,\ldots, 1, \lambda_n),
\end{eqnarray*} where we set $c_i=S\cdot D_i$ for $i=1,\ldots, n$. Since the $\lambda_i\in\Q$ are independent of $S$, it follows that the vectors $\{(1,0,\ldots, \lambda_1),\ldots (0,\ldots, 1, \lambda_n)\}$ span $f_{n+1}(\NS(X))$, from where the claim follows.

\end{proof}

\begin{rem}\label{NSXfree}
   The group $\NS(X)$ is finitely generated abelian.  The pullback $\iota^\star$ clearly factors through the maximal free abelian subgroup $\NS(X)/\NS(X)_{\tor}$.  
\end{rem}
	
\subsection{Some results about algebraic tori}	
\begin{ass}\label{Assumption}
From now on and for the rest of this article we will be working over a finite extension $k$ of the $p$-adic field $\Q_p$. We will denote by $\sO_k$ the ring of integers of $k$,  by $\mathfrak{m}_k$ the maximal ideal of $\sO_k$, and by $\kappa=\sO_k/\mathfrak{m}_k$ its residue field. 
\end{ass} 

\begin{lem}\label{lem:torus0} Let $f:T_1(k)\to T_2(k)$ be a homomorphism of split tori over $k$. If $\cok(f)$ is torsion of finite exponent, then it is finite. 
\end{lem}
\begin{proof}
	Consider the  exact sequence $0 \rightarrow A \rightarrow T_{2}(k)
 \rightarrow \cok(f) \rightarrow 0$, where $A$ is the image of the map $f$. For any $n \geq 1$, the short exact sequence induces the following long exact sequence:
 \begin{center}
 	 $0 \rightarrow A_n \rightarrow T_{2}(k)_n \rightarrow \cok(f)_{n} \rightarrow A/n \rightarrow T_{2}(k)/n  \rightarrow \cok(f)/n \rightarrow 0 $.
 \end{center}
 Since $k$ is a $p$-adic field, both $T_{1}(k)/n$ and $T_{2}(k)_{n}$ are finite, which in particular implies that $\cok(f)_{n}$ is finite for any $n \geq 1$. Moreover, by assumption $\cok(f) = \cok(f)_{m}$ for some integer $m $,  and hence $\cok(f)$ is finite.

   \end{proof}

\begin{prop}\label{prop:torus1} Let $f:T_1(k)\to T_2(k)$ be a  continuous homomorphism of split tori over $k$. If $\cok(f)$ is torsion, then it is finite. 
\end{prop}

\begin{proof}
 In view of \autoref{lem:torus0}, it is enough to show that  $\cok(f)$ is torsion of finite exponent. Since $k$ is a $p$-adic field, we have an isomorphism of topological groups, $k^{\times} \cong \mathbb{Z} \times \mathcal{O}_{k}^{\times}$ (see \cite[Proposition II. 5.7]{Neu}), where $\mathbb{Z}$ is equipped with the discrete topology and $\mathcal{O}_{k}^{\times}$, a profinite abelian group. Thus,
\begin{center}
    $T_1(k) =  (k^{\times})^{m} \cong \mathbb{Z}^{m} \times (\mathcal{O}_{k}^{\times})^{m}$ and $T_2(k) = (k^{\times})^{n} \cong  \mathbb{Z}^{n} \times (\mathcal{O}_{k}^{\times})^{n}$
\end{center} 
  for some integers $m, n \geq 1$. Here, we consider $T_1(k)$ and $T_2(k)$ as topological groups with the corresponding product topologies, and $f$ as a continuous group homomorphism. Consider the following commutative diagram of topological abelian groups, where the horizontal sequences are exact.
 \[
 \begin{tikzcd}
(\mathcal{O}_{k}^{\times})^{m}   \arrow[r, "f \circ   i_{2}"] \arrow[d, "i_{2}"] & \mathbb{Z}^{n} \times (\mathcal{O}_{k}^{\times})^{n}  \arrow[r] \arrow[d, "id"] & \cok(f \circ i_{2}) \arrow[r] \arrow[d, two heads, "\lambda"] & 0 \\
 	\mathbb{Z}^{m} \times (\mathcal{O}_{k}^{\times})^{m} \arrow[r, "f"]                 &  \mathbb{Z}^{n} \times (\mathcal{O}_{k}^{\times})^{n} \arrow[r]                & \cok(f) \arrow[r]                & 0,
 \end{tikzcd}
 \]
where $i_{2}$ denotes the inclusion into the second factor, $id$ denotes the identity map,  and $\lambda$ denotes the induced map on the cokernel of $f \circ i_{2}$ and $f$, respectively, such that the diagram is commutative. Note that the last two groups in both sequences are equipped with the quotient topology, and hence all maps in the diagram are continuous. Since $(\mathcal{O}_{k}^{\times})^{m}$ is profinite, there is no nontrivial continuous group homomorphism from $(\mathcal{O}_{k}^{\times})^{m}$ to $\mathbb{Z}^{n}$. This implies that we have an isomorphism of topological groups:
\[
\cok(f \circ i_{2}) \cong \mathbb{Z}^{n} \times G,
\]
where $G$ is a profinite abelian group. In fact,  $G \cong \frac{(\mathcal{O}_{k}^{\times})^{n}}{\operatorname{im}(f \circ i_{2})}$, where the quotient is profinite because the image of the map $f \circ i_{2}$ is a closed subgroup of  $ \mathbb{Z}^{n} \times  (\mathcal{O}_{k}^{\times})^{n}$ and can be viewed as a subgroup of $(\mathcal{O}_{k}^{\times})^{n}$.  We denote by $i'_{2}$ the inclusion of $G$ into  $ \cok(f \circ i_{2})$ via the second factor.
Consider the composite map $\lambda \circ i'_{2}$
\[
G \xrightarrow{i'_{2}} \cok(f \circ i_{2}) \xrightarrow{\lambda} \cok(f),
\]

which is continuous; we denote it by $\theta$. This gives rise to the following commutative diagram of topological groups:
\[
\begin{tikzcd}
0 \arrow[r] & G \arrow[r, "i_{1}"] \arrow[d, "\theta"'] & G \times \mathbb{Z}^{n} \arrow[r] \arrow[d, two heads, "\lambda"'] & \mathbb{Z}^{n} \arrow[r] \arrow[d, "0"] & 0 \\
0 \arrow[r] & \cok(f) \arrow[r, "="] & \cok(f) \arrow[r] & 0 \arrow[r] & 0,
\end{tikzcd}
\] 

where $i_{1}$ denote the inclusion into the first factor. Since $\lambda$ is surjective, it follows from the Snake Lemma that $\cok(\theta)$ is finitely generated. As $\cok(f)$ is torsion, it follows that $\cok(\theta)$ is finite. Since $\theta$ is a continuous group homomorphism, its kernel $\ker(\theta)$ is a closed subgroup, and hence profinite. This implies that the image of $\theta$ is a profinite abelian group.

As $\cok(f)$ is torsion, the image of $\theta$ is a torsion profinite abelian group. Hence it is of finite exponent (see \cite[Lemma 4.3.7]{RZ}), which shows that $\cok(f)$ is torsion of finite exponent. This completes the proof.

\end{proof}

\begin{prop}\label{prop:torus2} Let $f:T(k)\to T(k)$ be a continuous endomorphism of a split torus over $k$. If $\cok(f)$ is finite, then so is $\ker(f)$. 
\end{prop}
	\begin{proof}

We have an isomorphism of topological groups $T(k) = (k^{\times})^{r} \cong \mathbb{Z}^{r} \times (\mathcal{O}_{k}^{\times})^{r}$, where $r \in \mathbb{Z}_{\geq 0}$ and $\mathbb{Z}$ is discrete. Here, $\mathcal{O}_{k}^{\times}$ is a profinite abelian group, and we denote the topological group $(\mathcal{O}_{k}^{\times})^{r}$ by $G$. Consider the following diagram:

\[
\begin{tikzcd}
0 \arrow[r] & G \arrow[r, "i_{2}"] \arrow[d, "\tilde{f}"] & \mathbb{Z}^{r} \times G \arrow[r, "\pi_{1}"] \arrow[d, "f"] & \mathbb{Z}^{r} \arrow[r] \arrow[d, "f'"] & 0 \\
0 \arrow[r] & G \arrow[r, "i_{2}"] & \mathbb{Z}^{r} \times G \arrow[r, "\pi_{1}"] & \mathbb{Z}^{r} \arrow[r] & 0,
\end{tikzcd}
\]
where $i_{j}$ and $\pi_{j}$ denote the inclusion and projection maps for the $j$-th factor ($j = 1, 2$), respectively. Since $G$ is profinite and $\mathbb{Z}^r$ is discrete, $G$ admits no nontrivial continuous group homomorphisms into $\mathbb{Z}^r$. Thus, $\pi_{1} \circ f \circ i_{2} = 0$, which implies there exists a continuous group homomorphism $\tilde{f} \colon G \to G$ such that the first square commutes. Specifically, $\tilde{f}$ is defined by $\tilde{f} = \pi_2 \circ f \circ i_2$. This further induces a map $f'$ that makes the entire diagram commute.

By the Snake Lemma, we obtain the following long exact sequence:
\[
0 \longrightarrow \ker(\tilde{f}) \longrightarrow \ker(f) \longrightarrow \ker(f') \longrightarrow \cok(\tilde{f}) \longrightarrow \cok(f) \longrightarrow \cok(f') \longrightarrow 0.
\]

Since $\cok(f)$ is finite, it follows that $\cok(f')$ is also finite. Because $f'$ is an endomorphism of the free abelian group $\mathbb{Z}^r$, the finiteness of its cokernel implies that $f'$ is injective. This implies that $\ker(f') = 0$, and hence the above sequence yields an isomorphism $\ker(\tilde{f}) \cong \ker(f)$. Therefore, the original claim is equivalent to showing that $\ker(\tilde{f})$ is finite. Since $\cok(f)$ is finite, the above sequence also shows that $\cok(\tilde{f})$ is finite. Hence, we are reduced to showing the following claim.

\textbf{Claim}. Let $\theta : G \rightarrow G$ be a continuous group homomorphism with finite cokernel. Then $\ker(\theta)$ is finite.\\

\textbf{Proof}. As $G = (\mathcal{O}_{k}^{\times})^{r} \cong F \times (\mathbb{Z}_{p})^{m}$ (see \cite[Proposition II. 5.7]{Neu}) for some finite group $F$ and some integer $m$, consider the following composite map $\pi_{2} \circ \theta \circ i_{2}$:
\begin{center}
    $ \mathbb{Z}_{p}^{m} \xrightarrow{i_{2}} F \times \mathbb{Z}_{p}^{m}  \xrightarrow{\theta} F \times \mathbb{Z}_{p}^{m} \xrightarrow{\pi_{2}} \mathbb{Z}_{p}^{m}    $,
\end{center}
where $i_{2}$ (resp. $\pi_{2}$) is the inclusion (resp. projection) into the second factor. It is easy to see that the $\ker(\pi_{2} \circ \theta \circ i_{2})$ is finite if and only if $\ker(\theta)$ is finite. Moreover, since $\cok(\theta)$ is finite, it follows that $\cok(\pi_{2} \circ \theta \circ i_{2})$ is also finite. Hence, without loss of generality, we may assume $G \cong \mathbb{Z}_{p}^{m}$. In this case, we will show that $\ker(\theta) = 0$. Consider the image of the map $\theta$, say Im($\theta$), which is a closed subgroup of $G$ of finite index. It follows that Im($\theta$) is an open subgroup of $G$, and hence a topologically finitely generated profinite group (see \cite[Proposition 2.5.5]{RZ}). By the structure theorem (see \cite[Theorem 4.3.5]{RZ}), we have Im$(\theta) \cong \mathbb{Z}_{p}^{m'}$ for some integer $m' \leq m$. Since Im($\theta$) has finite index in $G$, it follows that $ m' = m$. We now have a continuous epimorphism $\theta : G = \mathbb{Z}_{p}^{m} \rightarrow \mathbb{Z}_{p}^{m} $. Such $\theta$ is an isomorphism (see \cite[Proposition 2.5.2]{RZ}), which implies the $\ker(\theta) = 0$. This completes the proof.

\end{proof}

\subsection{Structure of $SK_1(D)$ for snc schemes of dimension $1$}\label{SK1(D)}
In this subsection 
we extend the definition of $SK_{1}$ given in subsection \ref{k-theory-curves} to projective but not necessarily smooth curves. 
Before doing that we note that 
the group $SK_1(C)$ has also an interpretation as a Bloch's higher Chow group, as the next proposition shows. 
\begin{prop}\label{prop:classgroup-chow-motivic}
Let $C$ be a smooth projective curve over a $p$-adic field $k$. Then we have an isomorphism:
\begin{center}
    $SK_{1}(C) \cong \CH^{2}(C, 1) \cong H_{M}^{3}(C, \mathbb{Z}(2))$.
\end{center}
\end{prop}

\begin{proof}
The first isomorphism in the claim follows from \cite[Theorem 2.5]{Lan}, and the second isomorphism follows from  (\ref{motivic-higher-chow-iso}) (see also \cite[Appendix]{AS}).

\end{proof}

We now let $D$ be a projective curve over $k$. We define
\begin{equation}
    SK_{1}(D) := H_{M}^{3}(D, \mathbb{Z}(2)).
\end{equation}
 \autoref{prop:classgroup-chow-motivic} shows that this definition coincides with the group defined in \S\ref{k-theory-curves} for smooth projective curves. The next result provides the structure of $SK_{1}(D)/SK_{1}(D)_{\dv}$ for projective curves, that are not necessarily smooth (also see \autoref{cor:SK1structure}).

\begin{prop}\label{prop:SK1} Let $X$ be a smooth projective surface over $k$ and $D$ a simple normal crossing divisor on $X$. Suppose that $D$ has $r\geq 1$ irreducible components, and each irreducible component has a $k$-rational point.Then we have the following.

\begin{enumerate}
    \item  The composite map $\varepsilon_{0} := N \circ \pi^{*} \colon SK_{1}(D) \to (k^{\times})^{r}$ is surjective, where $\pi^{*}:SK_{1}(D) \rightarrow SK_{1}(\tilde{D})$ is the map induced by the normalization $\pi \colon \tilde{D} \to D$, and $N \colon SK_{1}(\tilde{D}) \to (k^{\times})^{r}$ is the norm map.

\item The group $H:=SK_1(D)/SK_1(D)_{\dv}$ fits into a short exact sequence \[0\to F_2\to H\to F_1\oplus (k^\times)^r\to 0,\] where $F_1, F_2$ are finite groups. 
\end{enumerate}

\end{prop} 

\begin{proof}

 If $D$ is smooth, this follows from \autoref{cor:SK1structure}. Assume $r \geq 2$. Let $\{ D_{i} \}_{1 \leq i \leq r}$ be the set of irreducible components of $D$, where each $D_{i}$ is a smooth curve. Let $\pi \colon \tilde{D}:= \coprod_{i=1}^{r} D_{i}  \to D$ be the normalization of $D$. We then have the following blow-up square:
\[
\begin{tikzcd}
E \arrow[r, hookrightarrow, "\Tilde{i}"] \arrow[d, "\pi'"'] & \Tilde{D} \arrow[d, "\pi"] \\
D_{\sing} \arrow[r, hookrightarrow, "i"'] & D,
\end{tikzcd}
\]
where $D_{\text{sing}}$ is the singular locus of $D$, $E$ is the exceptional fiber of the map $\pi$, and $i$ and $\tilde{i}$ denote the canonical inclusions of closed subschemes $D_{\sing}$ and $E$ respectively. The above blow-up square induces the following commutative diagram:
\[
\begin{tikzcd}[column sep=small, row sep=small]
\scriptsize
H_{M}^{2}(D_{\sing}, \mathbb{Z}(2)) \arrow[r] \arrow[d, "\pi'^{\ast}"'] & 
H_{M,c}^{3}(D - D_{\sing}, \mathbb{Z}(2)) \arrow[r] \arrow[d, "\cong"'] & 
H_{M}^{3}(D, \mathbb{Z}(2)) \arrow[r] \arrow[d, "\pi^{\ast}"] & 
H_{M}^{3}(D_{\sing}, \mathbb{Z}(2)) \arrow[d, "\pi'^{\ast}"'] \\
H_{M}^{2}(E, \mathbb{Z}(2)) \arrow[r] & 
H_{M,c}^{3}(\Tilde{D} - E , \mathbb{Z}(2)) \arrow[r] & 
H_{M}^{3}(\Tilde{D} , \mathbb{Z}(2)) \arrow[r] & 
H_{M}^{3}(E, \mathbb{Z}(2))
\end{tikzcd}
\]
where the horizontal rows are exact by  \autoref{lem:cohomology Ex-seq}. The second vertical map is an isomorphism since $\Tilde{D} - E \cong D - D_{\sing}$. As the dimension of both $D_{\sing}$ and $E$ is zero, the two rightmost groups in both sequences vanish by a dimension argument. Infact, for a field $k$, we have $H_{M}^{3}(\Spec(k), \mathbb{Z}(2)) \cong \CH^{2}(\Spec(k), 1)  = 0$. Hence, we obtain a short exact sequence
\begin{equation}\label{classgroup-snc-exact-1}
    0 \rightarrow \Ker(\pi^{\ast}) \rightarrow H_{M}^{3}(D, \mathbb{Z}(2)) \xrightarrow{\pi^{\ast}} H_{M}^{3}(\Tilde{D} , \mathbb{Z}(2)) \rightarrow 0.
\end{equation}
Furthermore, by five lemma, we have $H_{M}^{2}(E, \mathbb{Z}(2)) \twoheadrightarrow \Ker(\pi^{\ast})$.  Note that we have the following isomorphism.
\begin{equation}
H_{M}^{3}(\Tilde{D}, \mathbb{Z}(2)) \cong \bigoplus_{i=1}^{r} H_{M}^{3}(D_{i}, \mathbb{Z}(2)) \cong \bigoplus_{i=1}^{r} SK_{1}(D_{i}).
\end{equation}

\noindent \textbf{Proof of (1):} The map $\pi^*$ is just the pullback 
\begin{equation}
    SK_{1}(D) \cong H_{M}^{3}(D, \mathbb{Z}(2)) \xrightarrow{\pi^{*}} H_{M}^{3}(\tilde{D}, \mathbb{Z}(2)) \cong SK_{1}(\tilde{D}) \cong \bigoplus_{i=1}^{r} SK_{1}(D_{i}).
\end{equation}
For each $i$, let $N_{i} \colon SK_{1}(D_{i}) \to k^{\times}$ denote the norm map for the smooth curve $D_{i}$. Since $D_{i}(k) \neq \emptyset$, the map $N_{i}$ is surjective for each $i$ (see \ref{sec:SK1}). Defining $N = \bigoplus_{i=1}^{r} N_{i}$ yields the norm map
\[
SK_{1}(\tilde{D}) \cong \bigoplus_{i=1}^{r} SK_{1}(D_{i}) \xrightarrow{N} (k^{\times})^{r},
\]
which is also surjective. Consequently, the composite map
\[
\varepsilon_0 := N \circ \pi^{*} \colon SK_{1}(D) \xrightarrow{\pi^{*}} SK_{1}(\tilde{D}) \xrightarrow{N} (k^{\times})^{r}
\]
is surjective, which completes the proof of (1).

\noindent \textbf{Proof of (2):}  Note that $E = \coprod \Spec(k_{j})$ for $1 \leq j \leq n$, where each $k_{j}$ is a $p$-adic field, we have:
\[
H_{M}^{2}(E, \mathbb{Z}(2)) \cong \bigoplus_{j = 1}^{n} H_{M}^{2}(\Spec(k_{j}), \mathbb{Z}(2)) \cong \bigoplus_{j = 1}^{n} \CH^{2}(\Spec(k_{j}), 2).
\]
For each $j$, we have $\CH^{2}(\Spec(k_{j}), 2) \cong^{1} K^{M}_{2}(k_{j}) \cong^{2} F_{j} \oplus D_{j}$, where $F_{j}$ is a finite group and $D_{j}$ is a divisible group. Here the isomorphism $\cong^{1}$ follows from \cite[Theorem 5.1]{MVW}, and $\cong^{2}$ follows from \cite[Chapter VI Proposition 7.1]{Wei}.  Therefore by \eqref{classgroup-snc-exact-1}, we have the following short exact sequence:
\begin{equation}
    0 \rightarrow \Ker(\pi^{\ast}) \rightarrow H_{M}^{3}(D, \mathbb{Z}(2)) \xrightarrow{\pi^{\ast}} H_{M}^{3}(\Tilde{D} , \mathbb{Z}(2)) \rightarrow 0 
\end{equation}
where $\Ker(\pi^{\ast}) \cong F' \oplus D'$, for some finite group $F'$ and a divisible group $D'$. In addition, there are isomorphisms
\[
H_{M}^{3}(\Tilde{D}, \mathbb{Z}(2)) \cong \bigoplus_{i=1}^{r} H_{M}^{3}(D_{i}, \mathbb{Z}(2)) \cong \bigoplus_{i=1}^{r} SK_{1}(D_{i}) \cong \Big (\bigoplus_{i=1}^{r}SK_1(D_{i})_{\dv} \Big)\bigoplus F \bigoplus (k^{\times})^{r},
\]
where the last isomorphism follows from \autoref{cor:SK1structure} for some finite group $F$. Therefore, we obtain a short exact sequence:
\begin{equation}\label{pullback-sk1(D)}
0 \rightarrow F' \oplus D' \rightarrow SK_{1}(D) \xrightarrow{\pi^{\ast}} \Big (\bigoplus_{i=1}^{r}SK_1(D_{i})_{\dv} \Big)\bigoplus F \bigoplus (k^{\times})^{r}\rightarrow 0,
\end{equation}
where $F$ and $F'$ are finite groups and $D'$ is divisible. The desired claim can now be easily deduced from the equation (\ref{pullback-sk1(D)}) (see \autoref{lem:F+D}). This completes the proof.

\end{proof}

\subsection{Decomposable $(2,1)$-cycles} 
In the previous subsection we saw that for a smooth projective curve $C$ over $k$,  $SK_1(C)\simeq\CH^2(C,1)$. More generally, if $X$ is a smooth projective variety over $k$ there is an isomorphism 
\[\CH^2(X,1)\simeq H^1_{Zar}(X,\mathcal{K}_2),\] (\cite[Theorem 2.5]{Lan}) where $\mathcal{K}_2$ is the Zariski sheaf associated to the presheaf, $U\mapsto K_2(U)$, the second algebraic $K$-group of $U$.  The product structure of $K$-theory gives a homomorphism
\[\psi:\Pic(X)\otimes k^\times\to H^1_{Zar}(X,\mathcal{K}_2).\] 
It is standard to call an element of $\CH^2(X,1)$ \textit{decomposable}, if it lies in the image of the map
\[\bigoplus_{L/k \text{ finite}}\Pic(X_L)\otimes L^\times\xrightarrow{\bigoplus \psi_L} \bigoplus\CH^2(X_L,1)\xrightarrow{\sum_{L/k}N_{L/k}}\CH^2(X,1).\] 
When $X$ is a curve, every $(2,1)$-cycle is decomposable (see the proof of \autoref{prop:secsoffibration}). For a surface $X$ with $p_g(X)>0$ this is generally not expected to be true, but constructing indecomposable cycles is a rather nontrivial task. We refer to \cite{Mil,Sp} for some examples. Fortunately, the decomposable $(2,1)$-cycles will be enough for our purposes. 

\vspace{3pt}
 
\section{Structure of Suslin's Homology for surfaces over $p$-adic fields}\label{mainproofs}
In this section we prove \autoref{thm:main2} and \autoref{maincor}, which relate \autoref{mainconj} for smooth projective surfaces to its analogue for smooth quasi-projective surfaces (see \autoref{ourconj}).  We keep assuming we are in the set-up of Assumption \ref{Assumption}. We begin this section by establishing certain results on Suslin homology which will play a key role in the proof of the theorem.

\begin{lem}\label{lem:Sus-hom-surjection}
Let $X$ be a smooth projective surface over $k$, and let 
$u \colon U \inj X$ be a dense open immersion. Then the natural map 
$H^{\mathrm{\Sus}}_{0}(U) \rightarrow \CH_{0}(X)$ is surjective. Moreover, if the complement has dimension $0$, it is an isomorphism.
\end{lem}

\begin{proof}
Let $D := X - U$ be the closed subscheme of dimension at most one. By \autoref{lem:cohomology Ex-seq}, we have an exact sequence:
\[
H_{M}^{3}(D, \mathbb{Z}(2)) \rightarrow H_{M,c}^{4}(U, \mathbb{Z}(2)) \xrightarrow{u_{\ast}} H_{M}^{4}(X, \mathbb{Z}(2)) \rightarrow H_{M}^{4}(D, \mathbb{Z}(2)).
\]
Using the isomorphisms (\ref{motivic-higher-chow-iso}) and (\ref{compact-motivic-Sus-iso}), it is enough to show that the group $H_{M}^{4}(D, \mathbb{Z}(2))$ vanishes. Since motivic cohomology is nil-invariant, we may assume that $D$ is reduced. If $D$ is smooth, then 
\[
H_{M}^{4}(D, \mathbb{Z}(2)) \cong \CH^{2}(D, 0) \cong \CH^2(D) = 0,
\]
where the last equality holds as  $\dim(D) \leq 1$. If $\dim(D) = 0$, then $D$ is smooth and we are done. Moreover, $H_{M}^{3}(D, \mathbb{Z}(2)) \cong \CH^{2}(D, 1) = 0$, where the latter group vanishes again by a dimension argument. This proves that the map
$H^{\mathrm{\Sus}}_{0}(U) \rightarrow \CH_{0}(X)$ is an isomorphism if $\dim(D) = 0$. 

From now on assume that $\dim(D) = 1$. Let $D_{\sing}$ (resp. $D_{\sm}$) denote the singular locus (resp. smooth locus) of $D$. By \autoref{lem:cohomology Ex-seq}, we have an exact sequence:
\[
\cdots \to H^4_{M,c}(D_{\sm}, \mathbb{Z}(2)) \rightarrow H^4_M(D, \mathbb{Z}(2)) \rightarrow H^4_{M}(D_{\sing}, \mathbb{Z}(2)) \to \cdots
\]
Since $\dim(D_{\sing}) = 0$, the previous discussion implies that $H^4_{M}(D_{\sing}, \mathbb{Z}(2)) = 0$.

Also, by \cite[Remark 2.2]{Yam1}, we have 
\[
H^4_{M,c}(D_{\sm}, \mathbb{Z}(2)) \cong H^{M}_{-2}(D_{\sm}, \mathbb{Z}(-1)),
\]
and the latter group vanishes by \cite[Lemma 2.1 (2)]{Yam1}. This in particular shows that the group $H_{M}^{4}(D, \mathbb{Z}(2))$ vanishes, which completes the proof.

\end{proof}

The next result shows that the subgroup of $n$-torsion elements in the Suslin homology of a quasi-projective surface, $H^{\Sus}_{0}(U)_{n}$, is finite. An analogous mod-$n$ finiteness result for $H^{\Sus}_{0}(U)/n$ was shown in \cite[Theorem 1.5]{BK}. 

\begin{prop}\label{finiteness-n-torsion-sus} 
Let $X$ be a smooth projective surface over $k$, and  
$u \colon U \inj X$  a dense open immersion. Then for any integer $n\geq 1$, the group $H^{\Sus}_{0}(U)_{n}$ is finite.
\end{prop}

\begin{proof}
By \autoref{thm:compactly-support-suslin}, the claim is equivalent to showing that for every $n\geq 1$ the group $H^4_{M,c}(U, \mathbb{Z}(2))_{n}$ is finite. We now use the short exact sequence,

\begin{center}
    $0 \rightarrow H^3_{M,c}(U, \mathbb{Z}(2))/n \rightarrow H^3_{M,c}(U, \mathbb{Z}/n(2)) \rightarrow H^4_{M,c}(U, \mathbb{Z}(2))_{n} \rightarrow 0$
\end{center}
(see \cite[(2.1.1)]{Yam1}). It is enough to show that the middle group is finite. By \cite[Corollary 10.3]{GKR} (also  \cite[Proposition 3.2]{Yam1}), the group $H^3_{M,c}(U, \mathbb{Z}/n(2))$ injects into $H^3_{\text{ét},c}(U, \mathbb{Z}/n(2))$ where the latter group is compactly supported \'{e}tale cohomology, which is finite by \cite[Lemma 2.9]{Sai-Duality}. This completes the proof.

\end{proof}

\begin{lem}\label{lem:maindiagram} Let $X$ be a smooth projective surface over $k$ and $j:U\hookrightarrow X$ an open immersion whose complement has dimension one. Then there is a commutative diagram with exact rows and columns, 
    \[
\begin{tikzcd}
&     & \CH^2(X,1) \arrow[r] \arrow[d, "g"] & 0 \arrow[d] \\
&     & SK_1(X-U) \arrow[r, "\varepsilon"] \arrow[d, "f"] & T(k) \arrow[d, "\delta"] \\
	0 \arrow[r] & F^{2}(U) \arrow[r] \arrow[d, "\alpha"] & F^{1}(U) \arrow[r, "\alb_U"] \arrow[d, "\beta"] & Alb_{U}(k) \arrow[d, "\gamma"] \\
	0 \arrow[r] & F^{2}(X) \arrow[r]                 & F^1(X)  \arrow[r, "\alb_X"]    \arrow[d]             & Alb_{X}(k) \\
	 &                  & 0.   & 
\end{tikzcd}
\] 
Moreover, if the torus $T$ is split, then the map $\gamma$ is surjective. 
\end{lem}

\begin{proof}
The third and fourth rows as well as the last column are exact by definition. The exactness of the middle column  follows by \autoref{lem:cohomology Ex-seq}, \autoref{thm:compactly-support-suslin}, \autoref{lem:Sus-hom-surjection}, and the fact  $\ker(F^{1}(U) \rightarrow F^{1}(X)) = \ker(H_{0}^{sus}(U) \rightarrow \CH_{0}(X))$.

Lastly, suppose that the torus $T$ is split. Then we claim that the sequence $0\to T(k)\to Alb_U(k)\to Alb_X(k)\to 0$ is exact on the right. This follows from the long exact sequence in Galois cohomology,
$0\to T(k)\to Alb_U(k)\to Alb_X(k)\to H^1(k, T(\overline{k}))$.  Since we assumed that the torus $T$ is split, we may write $T\simeq\G_m^s$ for some integer $s\geq 0$, and hence $H^1(k, T(\overline{k}))=0$ by Hilbert Theorem 90. 
    
\end{proof}

\begin{rem}\label{rem:epsilonmap} Suppose that $X\setminus U=|D|$, where $D$ is a simple normal crossing divisor satisfying the assumptions of \autoref{prop:SK1}. In this case we claim that
the map $\varepsilon$ appearing in the diagram of  \autoref{lem:maindiagram} is the same as $j^{\ast}(k) \circ \varepsilon_{0}$, where $\varepsilon_{0}: SK_{1}(D) \rightarrow (k^{\times})^r$ is defined in Proposition \ref{prop:SK1}, and $j^{\ast}(k): (k^{\times})^{r} \rightarrow (k^{\times})^{r-n}$ is the map of tori obtained as the dual of  the inclusion of character groups $J \hookrightarrow \mathbb{Z}^{r}$ (see subsection \ref{genalbsection}). Indeed, by Serre's construction of $Alb_{U}$ \cite{Se}, the map $\varepsilon$ is defined as follows. Following the notation in loc.cit., we denote by $E = X \times_{\text{Alb}_{X}} \text{Alb}_{U}$, the principal fiber space with base $X$ and group $T$, and we denote by $E_{j}$, the principal fiber space with base space $X$ and group $\mathbb{G}_{m}$ associated to a character $j \in J = \text{Hom}(T, \mathbb{G}_{m})$. Let $h_{j}$ denote a rational section of $E_{j}$ such that $\text{div}(h_{j}) = [j]$, where $[j]:=\varphi(j)$ is the class of $j$ in $\text{Pic}^0(X)$ (see \autoref{sec:background} for the definition of $\varphi$). Let  $\alpha \in SK_{1}(D)$, we write $f(\alpha) = \sum n_{i}[u_{i}] \in H^{\text{sus}}_{0}(U)$, and then
\[
\varepsilon(\alpha) = \text{alb}_{U}(f(\alpha)) := \left( \prod_{i} \text{Norm} (h_{j}(u_{i})^{n_{i}}) \right)_{j \in S} \in (k^{\times})^{r-n},
\]
where $S$ denotes a canonical generating set $S=\{j_1,\ldots, j_{r-n}\}$ of $J$. By the Weil reciprocity law for $K_{2}$, it follows that $\varepsilon (\alpha) = j^{\ast}(k) \circ \text{Norm} \circ \pi^{\ast} (\alpha) = j^{\ast}(k) \circ \varepsilon_{0} (\alpha)$.
\end{rem}

\begin{theo}\label{thm:main2} Let $X$ be a smooth projective surface over $k$ and $U$ a nonempty open subvariety. Assume that the complement $D=X-U$ is either $0$-dimensional, or the support of a simple normal crossing divisor. Then the group $F^2(X_L)/F^2(X_L)_{\dv}$ is finite for every finite extension $L/k$ if and only if $F^2(U_{L})/F^2(U_{L})_{\dv}$ is finite for every finite extension $L/k$. 
 \end{theo}
 
 \begin{proof} We start with some preliminary reductions. 
For a smooth projective surface $X$ over a $p$-adic field $k$ it follows by \cite[Theorem 8.1]{CT1} that for every $n\geq 1$ the  group $\CH_0(X)_n$ is finite. Since the quotient $F^2(X)/F^2(X)_{\dv}$ is a direct summand of $F^2(X)$, it can be thought of as a subgroup of $\CH_0(X)$. Thus,  the group $F^2(X)/F^2(X)_{\dv}$ is finite if and only if it is torsion of  finite exponent. Similarly,  \autoref{finiteness-n-torsion-sus} gives that the same holds for smooth quasi-projective surfaces and Suslin's homology $H_0^{\Sus}(U)$. This together with a simple norm argument (see for example \cite[Lemma 3.3]{Gaz24}) gives that if the group $F^2(X_L)/F^2(X_L)_{\dv}$ is finite for some finite extension $L/k$ (resp. $F^2(U_L)/F^2(U_L)_{\dv}$), then the same is true for $F^2(X)/F^2(X)_{\dv}$ (resp. $F^2(U)/F^2(U)_{\dv}$). This allows us to prove the theorem after extending to a finite extension, 
and hence we may assume all the following are true: 
\begin{itemize}
\item $U(k)\neq\emptyset$.
\item Each irreducible component of $X-U$ has a $k$-rational point (and hence the maximal subtorus $T$ of the semi-abelian variety $Alb_U$ is split). 
\item The absolute Galois group $\Gal(\overline{k}/k)$ acts trivially on the geometric N\'{e}ron-Severi group $NS(X_{\overline{k}})$. 
\end{itemize}

Assume first that $X-U$ has dimension $0$. As discussed in \autoref{genalbsection}, in this case $T=0$ and $Alb_U=Alb_X$.
Moreover, it follows from \autoref{lem:Sus-hom-surjection} that the map $\beta:H_0^{\Sus}(U)\to \CH_0(X)$ is an isomorphism. Since we assumed that $U$ has a $k$-rational point, this descends to an isomorphism $F^1(U)\xrightarrow{\beta} F^1(X)$. 
We conclude that in this case $F^2(U) \cong F^2(X)$, and hence the theorem holds trivially. 

From now on assume that $X-U$ is the support of a simple normal crossing divisor $D=D_1+\cdots+D_r$. We denote by $s=n-r$ the dimension of the maximal subtorus of $Alb_U$. We consider the commutative diagram of \autoref{lem:maindiagram}, 
 \[
\begin{tikzcd}
&     & \CH^2(X,1) \arrow[r] \arrow[d, "g"] & 0 \arrow[d] \\
&     & SK_1(D) \arrow[r, "\varepsilon"] \arrow[d, "f"] & T(k) \arrow[d, "\delta"] \\
	0 \arrow[r] & F^{2}(U) \arrow[r] \arrow[d, "\alpha"] & F^{1}(U) \arrow[r, "\alb_U"] \arrow[d, "\beta"] & Alb_{U}(k) \arrow[d, "\gamma"] \\
	0 \arrow[r] & F^{2}(X) \arrow[r]                 & F^1(X)  \arrow[r, "\alb_X"]    \arrow[d]             & Alb_{X}(k) \arrow[d] \\
	 &                  & 0   & 0.
\end{tikzcd}
\]  

 Applying the Snake Lemma with respect to 
two rightmost columns yields an exact sequence 
\begin{equation}\label{eq3}
\ker(\varepsilon)\xrightarrow{f} F^2(U)\xrightarrow{\alpha} F^2(X)\to\cok(\varepsilon)\to\cok(\alb_U)\to\cok(\alb_X)\to 0.
\end{equation}

\textbf{Claim 1:} The map $\varepsilon: SK_1(D)\to T(k)$ has finite cokernel.

\textbf{Proof.} 
Recall from \autoref{rem:epsilonmap} that the map $\varepsilon$ in the above commutative diagram coincides with the composition $\varepsilon=j^\star(k)\circ\varepsilon_0$. It follows from \autoref{prop:SK1}(1) that the map $\varepsilon_0$ is surjective. Moreover, it is clear that over the algebraic closure $\overline{k}$ so is $j^\star(\overline{k})$, since $j^\star:\G_m^{r}\to T$ is a surjective map of varieties. In particular,  Claim 1 is true over $\overline{k}$.

The next step is to show that over the base field $k$, $\cok(\varepsilon)$ is torsion. Let $x\in T(k)$. We may assume that $x$ is not a torsion element. 
Since $\varepsilon_{\overline{k}}$ is surjective, there exists a finite extension $L/k$ and some $y\in SK_1(D_L)$ such that $\res_{L/k}(x)=\varepsilon_L(y)$, where $\varepsilon_L:SK_1(D_L)\to T_L(L)$ is the base change to $L$ and $\res_{L/k}:T(k)\to T_L(L)$ is the restriction map. 
Since the field $L$ is $p$-adic, the torus $(L^\times)^s$ has no nonzero divisible elements and hence the map $\varepsilon_L$ factors through $H_L=SK_1(D_L)/SK_1(D_L)_{\dv}$. It follows by \autoref{prop:SK1} that the group $H_L$ fits into a short exact sequence $0\to F_2\to H_L\to F_1\oplus (L^\times)^r\to 0$, where $F_1, F_2$ are finite groups. Since we assumed that $x$ is nontorsion, we may even assume that $y\in (L^\times)^r$ by multiplying $x$ and $y$ if necessary by $N=|F_1|\cdot|F_2|$. Since the base field is perfect, the generalized Albanese map behaves well with finite base change, and we obtain a commutative diagram
\[
\begin{tikzcd}
&   (L^\times)^r \arrow[r, "\varepsilon_L"] \arrow[d, "N_{L/k}"]  & (L^\times)^s \arrow[d, "N_{L/k}"] \\
&  (k^\times)^r \arrow[r, "\varepsilon"]   & (k^\times)^s.
\end{tikzcd}
\] Let $n=[L:k]$. We then have,
\[\varepsilon(N_{L/k}(y))=N_{L/k}(\varepsilon_L(y))=N_{L/k}(\res_{L/k}(x))=x^n.\] This yields that $x^n\in\img(\varepsilon)$, verifying that $\cok(\varepsilon)$ is torsion. 

In fact, the above computation shows that the map $(k^\times)^r\xrightarrow{\varepsilon} (k^\times)^s$ has torsion cokernel. It then follows by \autoref{prop:torus1} that this cokernel is in fact finite. This completes the proof of Claim 1.  

The finiteness of $\cok(\varepsilon)$ together with the exact sequence \eqref{eq3} imply that we have an exact sequence 
\begin{equation}\label{eq4}
    \ker(\varepsilon)\to F^2(U)\xrightarrow{\alpha} F^2(X)\to F\to 0,
\end{equation} where $F$ is a finite group.

Now suppose that the group $F^2(U)/F^2(U)_{\dv}$ is finite. It then follows by \autoref{lem:F+D} that the group $F^2(X)/F^2(X)_{\dv}$ is torsion of finite exponent, and hence finite. This completes the proof of the converse direction of our theorem. 

We next focus on the forward direction, where we assume that $F^2(X)/F^2(X)_{\dv}$ is a finite group. The map $\varepsilon$ factors through $\img(f)$, which by exactness is isomorphic to $SK_1(D)/\img(g)$. Denote by $\overline{\varepsilon}:SK_1(D)/\img(g)\to (k^\times)^s$ the induced map with kernel $\ker(\overline{\varepsilon})\cong \ker(\varepsilon)/\img(g)\cap\ker(\varepsilon)$. Moreover, let $\pi:SK_1(D)\twoheadrightarrow SK_1(D)/\img(g)$ be the projection. The sequence \eqref{eq4} can be rewritten as 
\begin{equation}\label{eq5}
    0\to\ker(\overline{\varepsilon})\to F^2(U)\to F^2(X)\to F\to 0. 
\end{equation} Thus, using \autoref{lem:F+D}, the finiteness of $F^2(U)/F^2(U)_{\dv}$ can be reduced to showing the following claim: \\
\vspace{1mm}
\textbf{Claim 2:} The group $\ker(\overline{\varepsilon})$ is torsion  of finite exponent by divisible. 

\textbf{Proof.}
Consider the short exact sequence 
\[0\to F_2\to H\to F_1\oplus (k^\times)^r \to 0,\] where $H=SK_1(D)/SK_1(D)_{\dv}$.
Here we recall that $F_1, F_2$ are finite groups and $r$ is the number of irreducible components of the divisor $D$. 
Since quotients of divisible groups are divisible and the field $k$ is $p$-adic, it follows that $\pi(SK_1(D)_{\dv})\subset\ker(\overline{\varepsilon})$ is a divisible subgroup of $\ker(\overline{\varepsilon})$.

It remains to show that the kernel of the induced map 
\[\overline{\varepsilon}:\frac{SK_{1}(D)}{\img(g) + SK_1(D)_{\dv}}\to (k^\times)^s\] is torsion of finite exponent. Since $N\ker(\overline{\varepsilon})\subset\ker(N\circ\overline{\varepsilon})$, for all $N\geq 1$, it is enough to show that $\ker(N\circ\overline{\varepsilon})$ is torsion of finite exponent for some $N\geq 1$. Set $N=|F_1||F_2|$. Consider the map $N\varepsilon: SK_1(D)/SK_1(D)_{\dv}\to (k^\times)^s$. Since $Ny=0$, for all $y\in F_2$, this map factors through $(SK_1(D)/SK_1(D)_{\dv})/F_2\simeq F_1\oplus k^r$. Lastly, since $|F_1|$ divides $N$, it follows that $N\varepsilon$ factors through the toric piece, $(k^\times)^r$, of $H$. This argument shows that we may reduce to the case when $F_1=F_2=0$.

Recall from \autoref{genalbsection} that  $s=r-n$, where $n$ is the size of a maximal $\Z$-linearly independent subset of $\{[D_1],\ldots, [D_r]\}$ when considered as a subgroup of $\NS(X)$.  
Let us assume for simplicity that the N\'{e}ron-Severi group  $\NS(X_{\overline{k}})=\NS(X)$ is torsion-free. Consider the homomorphism $\Pic(X)\otimes k^\times\to\CH^2(X,1)$. Since $\NS(X)$ is free we have a decomposition 
$\Pic(X)\simeq\Pic^0(X)\oplus\NS(X)$, and hence an induced homomorphism $\NS(X)\otimes k^\times\to\CH^2(X,1)$. We consider the composition
\[\rho:\NS(X)\otimes k^\times\to\CH^2(X,1)\xrightarrow{g} SK_1(D)\twoheadrightarrow \bigoplus_{i=1}^r(V(D_i)\oplus k^\times)\twoheadrightarrow (k^\times)^r.\] 
Then this is precisely the map obtained by applying $\otimes k^\times$ to the pullback 
\[\iota^\star:\NS(X)\to\bigoplus_{i=1}^r\NS(D_i)\] induced by the closed immersion $\iota:D\hookrightarrow X$. It follows from \autoref{prop:NSXcomputations} that $\img(\iota^\star)\cong\Z^n$, and hence the map 
$\overline{\varepsilon}: \frac{SK_{1}(D)}{\img(g) + SK_1(D)_{\dv}}\to (k^\times)^s$ factors through an $r-n=s$ dimensional torus. 
Since we already know from Claim 1 that $\cok(\overline{\varepsilon})$ is finite, Claim 2 follows from \autoref{prop:torus2}. 

If $\NS(X_{\overline{k}})=\NS(X)$ is not torsion-free, then we can imitate the above argument by considering the surjection $\Pic(X)\to \NS(X)/\NS(X)_{\tor}\to 0$, which splits, and using \autoref{NSXfree}.

 \end{proof}
   For a smooth projective variety $X$ over a $p$-adic field $k$, S. Saito and Sujatha (\cite[p.~409]{SSu}) showed that the cokernel of the Albanese map $F^1(X)\xrightarrow{\alb_X}Alb_X(k)$ is finite. As a byproduct of the proof of \autoref{thm:main2} we obtain an analog of this for quasi-projective surfaces under some assumptions.

\begin{cor}\label{cor:genalb} Let $X$ be a smooth projective surface over $k$ and $D$ a simple normal crossing divisor on $X$. Let $U=X-\supp(D)$. Suppose that $U(k)\neq\emptyset$ and that the maximal subtorus of $Alb_U$ is split.  Then the Albanese map $\alb_U: F^1(U)\to Alb_U(k)$ has finite cokernel. 
\end{cor}
\begin{proof}
Consider the exact sequence \eqref{eq3}. It follows by Claim 1 in the proof of \autoref{thm:main2} that under these assumptions $\cok(\varepsilon)$ is finite. Moreover, the same is true for  $\cok(\alb_X)$ (\cite[p.~409]{SSu}). The result then follows by exactness. 

\end{proof}

\begin{rem}
    We expect that the assumptions of \autoref{cor:genalb} can potentially be removed. One could go further and explore quasi-projective analogs of the main result of \cite{Kai}. In this article, Kai proved that for a smooth projective variety $X$ over a $p$-adic field $k$ admitting a smooth projective model $\mathcal{X}$ over $\Spec(\mathcal{O}_k)$ whose Picard scheme $\Pic_{\mathcal{X}/\mathcal{O}_k}$ is smooth, the cokernel of $\alb_X$ is Pontryagin dual to  $\cok[\Pic(X_{\overline{k}})\to \NS(X_{\overline{k}})]$. The authors hope to explore this in a future paper. 
\end{rem}

\autoref{thm:main2} suggests that an analogue of Colliot-Thélène's conjecture must hold for smooth quasi-projective surfaces over \( p \)-adic fields. 
 This urges us to propose the following:

\begin{conj}\label{mainconj-quasiprojective}
Let \( U \) be a smooth quasi-projective surface over a finite extension \( k \) of the \( p \)-adic field \( \Q_p \). Then the kernel \( F^{2}(U) \) of the generalized Albanese map admits a decomposition \( F^{2}(U) \cong F \oplus D \), where \( F \) is a finite group and \( D \) is a divisible group.
\end{conj}

It is clear that if \autoref{mainconj-quasiprojective} holds, then \autoref{mainconj} is true for smooth projective surfaces. As a consequence of \autoref{thm:main2}, we will now show that \autoref{mainconj} for smooth projective surfaces also implies \autoref{mainconj-quasiprojective}.

 \begin{cor}\label{maincor}
     \autoref{mainconj} holds for smooth projective surfaces if and only if \autoref{mainconj-quasiprojective} holds.
 \end{cor}

 \begin{proof}
   
As mentioned above, we only need to prove the ``only if'' part. Assume that \autoref{mainconj} holds. Let \( U \) be any smooth quasi-projective surface over \( k \), and let \( Y \) be a projective variety such that $U$ is an open subvariety of $Y$. By resolution of singularities, there exists a proper birational morphism \( X \rightarrow Y \), where \( X \) is smooth. In fact, since \( Y \) is projective, the map \( X \rightarrow Y \) can be taken to be projective (see \cite[Corollary 3.22]{Kol}). Since \( U \) is smooth, it can be considered as an open subvariety of \( X \), where \( X \) is a smooth projective surface over \( k \). Indeed by performing suitable blow-ups, we may assume that \( X - U \) is the support of a simple normal crossing divisor (see \cite[Corollary 0.4]{CJS}). 
Furthermore, similarly to the proof of \autoref{thm:main2}, by base change to a finite extension \( L/k \), we may assume that $U(L)\neq\emptyset$, each irreducible component of $X-U$ has a $L$-rational point, the maximal subtorus $T$ of the semi-abelian variety $Alb_U$ is split and the Galois group $\Gal(\overline{L}/L)$ acts trivially on the N\'{e}ron-Severi group $\NS(X_{\overline{L}})$. 
Now applying  \autoref{thm:main2} to the pair ($U_{L}, X_{L}$), the group
$F^2(X_{L})/F^2(X_{L})_{\dv}$ is finite if and only if $F^2(U_{L})/F^2(U_{L})_{\dv}$ is finite. Since $F^2(X_{L})/F^2(X_{L})_{\dv}$ is finite by assumption, it follows that $F^2(U_{L})/F^2(U_{L})_{\dv}$ is also finite. We now apply a standard norm argument (see for example \cite[Lemma 3.3]{Gaz24}) to conclude that $F^2(U)/F^2(U)_{\dv}$ is torsion of finite exponent. Finally, using \autoref{finiteness-n-torsion-sus}, we conclude that $F^2(U)/F^2(U)_{\dv}$ is finite, which concludes the proof.   

\end{proof}

\section{New evidence for Colliot-Th\'{e}l\`{e}ne's Conjecture}\label{sec:5}
\subsection{Proof of \autoref{thm:main1intro}}\label{sec:CTconj}
 We start with the following lemma. 

\begin{lem} Let $X, Y$ be smooth projective varieties over a perfect field $k$ and let $\phi:X\to Y$ be a birational morphism. Suppose that $X(k), Y(k)\neq\emptyset$. Then there is an isomorphism $F^2(X)\simeq F^2(Y)$. 
\end{lem}
\begin{proof}
Let $\phi_\star:\CH_0(X)\to \CH_0(Y)$ be the induced pushforward, which restricts to a homomorphism $\phi_\star:F^1(X)\to F^1(Y)$. It follows by \cite[Proposition 6.3]{CT3} that $F^1(X)\simeq F^1(Y)$. Moreover, the Albanese variety of a smooth projective variety is a birational invariant. Thus, we have a commutative diagram with exact rows,
\[
\begin{tikzcd} 0\arrow[r] & F^2(X) \arrow[r] \arrow[d, "\pi_\star"] &
	  F^{1}(X) \arrow[r, "\alb_X"] \arrow[d, "\pi_\star"] & Alb_{X}(k) \arrow[d, "\pi_\star"] \\
	     0\arrow[r] & F^2(Y) \arrow[r] &   F^1(Y)  \arrow[r, "\alb_Y"]                & Alb_{Y}(k) .
\end{tikzcd} 
\] 
Here the leftmost vertical map is induced from exactness. 
Since the second and third vertical maps are isomorphisms, so is the first one. 
    
\end{proof}

We are now ready to prove \autoref{thm:main1intro}, which we restate here. 

\begin{theo}\label{thm:main1} Let $X, Y$ be smooth projective surfaces over a $p$-adic field $k$. Suppose there is a generically finite rational map $\pi: X\dashrightarrow Y$. If \autoref{mainconj} is true for $X_L$ for every finite extension $L/k$, then it is true for $Y$.  The same result holds true if the map $\pi$ is defined over a finite extension of the base field. 
\end{theo}

\begin{proof} 
Let $U\subset X, V\subset Y$ be dense open subsets such that $\pi:U\to V$ is a finite morphism of degree $n$. We first want to reduce to the case when the complements $X-U, Y-V$ are either $0$-dimensional or support of simple normal crossing divisors. 
By using a sequence of blow-ups (\cite[Corollary 0.4]{CJS}), we obtain a variety $\Tilde{Y}$ birational to $Y$ and an open $\Tilde{V}$ in $\Tilde{Y}$ such that $\Tilde{Y}-\Tilde{V}$ is  the support of a simple normal crossing divisor and $U\xrightarrow{\pi} \Tilde{V}$ is a finite map. Repeating the argument for the pair $X, U$ we  obtain the desired reduction. 

Similarly to the proof of \autoref{thm:main2}, since $X, Y$ are smooth surfaces over $k$, it is enough to prove the theorem after we base change to a finite extension. Thus, we may assume $U(k), V(k)\neq\emptyset$, and that the rational map $\pi$ is defined over $k$. 
Since the group $F^2(X)/F^2(X)_{\dv}$ is finite by assumption, it follows by the converse direction of \autoref{thm:main2} that so is the group $F^2(U)/F^2(U)_{\dv}$. We will show that $F^2(V)/F^2(V)_{\dv}$ is torsion of finite exponent (and hence finite by \autoref{finiteness-n-torsion-sus}). Then, the finiteness of $F^2(Y)/F^2(Y)_{\dv}$ will follow from the forward direction of \autoref{thm:main2}. 
We will prove that $F^2(V)/F^2(V)_{\dv}$ is torsion of finite exponent by using a pushforward and pullback argument along the lines of \cite[Theorem 1.2]{GL24}.  

Since $\pi:U\to V$ is a finite surjective morphism between equidimensional varieties, it induces pushforward and pullback maps on Suslin homology, 
\[\pi_\star: H_0^\Sus(U)\to H_0^\Sus(V),\;\;\;\;\;\;\;\pi^\star: H_0^\Sus(V)\to H_0^\Sus(U),\] such that $\pi_\star\circ\pi^\star=n$, where $n$ is the degree of the map $\pi$.  
One can see this by using the interpretation of $H_0^\Sus$ as motivic cohomology as in \autoref{thm:compactly-support-suslin} (see for example \cite[p.~5]{Yam1}).  
Since we assumed $U(k), V(k)\neq\emptyset$, these maps clearly reduce to homomorphisms on the degree $0$ subgroups, 
\[\pi_\star: F^1(U)\to F^1(V),\;\;\;\;\;\;\;\pi^\star: F^1(V)\to F^1(U).\] 
We claim that $\pi$ also induces homomorphisms at the level of  Albanese varieties, 
\[\pi_\star: Alb_U\to Alb_V,\;\;\;\;\;\;\;\pi^\star: Alb_V\to Alb_U.\]
Since $U(k)\neq\emptyset$, $Alb_U$ is universal for morphisms from $U$ to semi-abelian varieties. Thus, the map $U\xrightarrow{\pi} V\xrightarrow{\alb_V} Alb_V$ factors through $Alb_U$. This gives the map $\pi_\star$. The pullback is constructed in \cite[p.~5]{SpSz}, which we now review. The graph of the finite map $U\xrightarrow{\pi}V$ induces a finite correspondence $Z\hookrightarrow V\times U$, which in turn induces the composition
\[V\to\Sym^n U\xrightarrow{\alb_U}\Sym^n Alb_U\xrightarrow{+} Alb_U.\] Here the first map is induced by $\pi$, and the last map is addition in the algebraic group $Alb_U$. By universality of the Albanese variety, this composition factors through a homomorphism $\pi^\star:Alb_V\to Alb_U$. 
We have a commutative diagram  
 \[
\begin{tikzcd}
	  F^{1}(V) \arrow[r, "\alb_V"] \arrow[d, "\pi^\star"] & Alb_{V}(k) \arrow[d, "\pi^\star"] \\
	        F^1(U)  \arrow[r, "\alb_U"]                & Alb_{U}(k) .
\end{tikzcd} 
\]  
Since $\pi^\star: Alb_V(k)\to Alb_U(k)$ is a group homomorphism, it follows that if $z\in \ker(\alb_V)$, then $\pi^\star(z)\in\ker(\alb_U)$, and hence we get an induced homomorphism
\[\pi^\star: F^2(V)\to F^2(U).\]
Define $F^3(V):=\pi_\star(F^2(U))$, so that we have a filtration $H_0^\Sus(V)\supset F^1(V)\supset F^2(V)\supset F^3(V)\supset 0$. 
We claim that the quotient $F^2(V)/F^3(V)$ is $n$-torsion. For, 
\[\pi_\star\circ\pi^\star(F^2(V))=nF^2(V)\subset  \pi_\star(F^2(U))=F^3(V).\]

Next, let $D=\pi_\star(F^2(U)_{\dv})$, which is a divisible group. We have a commutative diagram with exact rows and surjective vertical maps, 
\[\begin{tikzcd}
0 \ar{r} & F^2(U)_{\dv}\ar{r}\ar{d}{\pi_\star} & F^2(U) \ar{r}\ar{d}{\pi_\star} & F^2(U)/F^2(U)_{\dv} \ar{r}\ar{d}{\pi_\star} & 0 \\
0 \ar{r} & D\ar{r} & F^3(V) \ar{r} & F^3(V)/D \ar{r} & 0
\end{tikzcd}.\]  Since the group $F^2(U)/F^2(U)_{\dv}$ is finite, the same is true for $F^3(V)/D$. Lastly, consider the short exact sequence $0\to D\to F^2(V)\to F^2(V)/D\to 0$. Since $F^2(V)/F^3(V)$     is torsion of finite exponent and $F^3(V)/D$ finite, it follows that $F^2(V)/D$  is torsion of finite exponent. This implies that $D=\pi_\star(F^2(U)_{\dv})$ is the maximal divisible subgroup of $F^2(V)$, and hence the quotient $F^2(V)/F^2(V)_{\dv}$ is torsion of finite exponent as desired. 

\end{proof} 

\subsection{Examples}\label{sec:exs}

In this section we use \autoref{thm:main1} to obtain evidence for Colliot-Th\'{e}l\`{e}ne's \autoref{mainconj} for many new classes of surfaces. We recall the following two definitions. 

\begin{defn}
    Let $X$ be a smooth projective surface over a $p$-adic field $k$. We say that $X$ is geometrically dominated by a product of curves $C_1\times C_2$, if over the algebraic closure $\overline{k}$ there is a dominant rational map $C_1\times C_2\dashrightarrow X_{\overline{k}}$. 
\end{defn}
Note that since $X$, $C_1\times C_2$ are both surfaces, the above definition is equivalent to having a generically finite rational map $C_1\times C_2\dashrightarrow X$ (see for example \cite[Exercise II.3.22]{RH}).

\begin{defn}\label{reddef} Let $A$ be an abelian variety over a $p$-adic field $k$. Let $\mathcal{A}$ be the N\'{e}ron model of $A$ and $\mathcal{A}_s:=\mathcal{A}\times_{\Spec(\sO_k)}\Spec(\kappa)$ its special fiber. We say that $A$ has a mixture of good ordinary and split multiplicative reduction, if the connected component $\mathcal{A}_s^0$ containing the identity element of $\mathcal{A}_s$ is a semi-abelian variety over the finite residue field $\kappa$ that fits into an exact sequence
    \[0\to T\to \mathcal{A}_s^0\to B\to 0,\] where $B$ is an ordinary abelian variety  and $T$ is a split torus. 
\end{defn}

\autoref{thm:main1} and the results of Raskind and Spiess (\cite[Corollary 4.5.7]{RS})  yield the following Corollary.

\begin{cor}\label{cor:newevidence}
    Let $X$ be a smooth projective surface over a $p$-adic field $k$. Suppose that $X$ is geometrically dominated by a product of curves $C_1\times C_2$ and the Jacobians $J_1, J_2$ of $C_1, C_2$ have a mixture of good ordinary and split multiplicative reduction. Then \autoref{mainconj} is true for $X$. 
\end{cor}

\begin{rem}\label{rem:reduction} \autoref{cor:newevidenceintro} can be thought of as following the general pattern that several important conjectures about algebraic cycles can be verified for varieties that are dominated by products of curves. Another example is the Tate conjecture for divisors (see \cite[Proposition 0.1]{Schoen}). We expect that this result can potentially be extended to higher dimensions. The real difficulty going forward is to remove the assumption on the reduction type of the Jacobians. \cite[Theorem 1.2]{GL} is the only case where some supersingular reduction factors are included. However, extending this even to a product $E_1\times E_2$ of two elliptic curves both having good supersingular reduction is a subtle problem. Nevertheless, we note that for an abelian variety $A$ over an algebraic number field $F$ a conjecture often attributed to Serre predicts that the places of good ordinary reduction are  plentiful. More precisely, it predicts  that the set of finite places $v$ of $F$ such that the base change $A_v$ to the completion $F_v$ has good ordinary reduction has positive density among all finite places. The case of elliptic curves is classical (\cite{Se3}). More recently the conjecture has been established for abelian surfaces (\cite{Og, Sw}). See also \cite{Fi} for some evidence for abelian threefolds. In higher dimensions this is known for example when $A$ has potential complex multiplication (\cite[Remark 12]{Fi}). 
\end{rem}

Next, we recall several classes of surfaces that are known to be geometrically dominated by products of curves. 
The most general case is the following. 

\subsubsection{Product-Quotient Surfaces}
Let $X=C_1\times C_2$ be a product of two smooth projective curves over $k$. Let $G$ be a finite group acting faithfully on $X$. Consider a resolution of singularities $Y$ of the quotient $X/G$. When $Y$ is minimal, it is usually referred to as \textit{product-quotient surface}. Such surfaces can be split into the following two different types (see \cite[Definitions 2.3, 2.4]{CP}). 

\begin{enumerate}
    \item[(i)] Suppose that $G$ is a subgroup of both $\Aut(C_1), \Aut(C_2)$ acting diagonally on $C_1\times C_2$. Then we say that $Y$ is a product-quotient surface of \textit{unmixed type}. 
    \item[(ii)] Let $X=C\times C$ and $G\leq \Z/2\times\Aut(C)$. Assume that there are elements of $G$ that exchange the two factors of $X$. Then we say that $Y$ is a product-quotient surface of \textit{mixed type}.  
\end{enumerate}

\begin{rem}\label{rem:resolutionsings} Finite group actions of unmixed type is the simpler of the two types in the sense that the singularities of the quotient $(C_1\times C_2)/G$ are all cyclic quotient singularities, whereas singularities in the mixed case are generally rather complicated. A very special case of a Type (i) surface is the Kummer surface $X=\Kum(A)$ associated to an abelian surface $A$, which is the minimal resolution of the quotient $A/\langle -1\rangle$. In the case $A=E_1\times E_2$ \autoref{mainconj} was proved in \cite[Theorem 1.2]{GL24} using a simpler version of \autoref{thm:main1}. Namely, if $A'$ is the blow-up of $A$ along the closed subscheme $A[2]$, then there is a commutative square 
\[\begin{tikzcd} A' \ar{r}{\sim} \ar{d}{\pi'} & A \ar{d}{\pi} \\
X \ar{r}{\sim} & A/\langle -1\rangle  
\end{tikzcd},\]
with the horizontal maps birational and the vertical maps finite morphisms of degree $2$. Thus, one can use the push-forward of the finite morphism $A'\xrightarrow{\pi'} X$, without having to work with open subvarieties. We note that such a simplified version of the argument cannot work for a general product-quotient surface of Type (i). Even though a cyclic-quotient singularity is the simplest type of singularity, Kollar (\cite[Example 2.30]{Kol}) has showed that when the order of the group $G$ is divisible by a prime $p\geq 5$, then in general there is no commutative square 
\[\begin{tikzcd} X' \ar{r}{\sim} \ar{d}{\pi'} & X \ar{d}{\pi} \\
Y \ar{r}{\sim} & X/G
\end{tikzcd},\] with $X', Y$ both smooth. However, when the group $G$ is of the form $G\cong(\Z/2)^n\times(\Z/3)^m$ for some $n,m\geq 0$, then such a diagram does exist. In fact, in this case the resolution of singularities of $(C_1\times C_2)/G$ can be obtained by a chain of blow-ups at closed points (\cite[Claim 2.29.1]{Kol}).  
\end{rem}

\subsubsection{Isotrivial Fibrations}\label{sec:isotrivial}
 We recall the following definition. 
\begin{defn}
    Let $Y\xrightarrow{\pi}C$ be a fibration over a smooth projective curve $C$ of genus $g\geq 0$. We say that $Y$ is isotrivial if all smooth fibers of $\pi$ are isomorphic. 
\end{defn}

 It is well-known (\cite[(1.1)]{Serrano}) that over $\overline{k}$ every isotrivial fibration is birational to a product-quotient surface $(C_1\times C_2)/G$ of unmixed type. This is a very wide class of surfaces. We next focus in particular to the $K3$ surfaces that happen to have a structure of an isotrivial fibration. For isotrivial surfaces with $p_g=q=1$ (resp. $p_g=q=2$) we refer to \cite{P, MP} (resp. \cite{CP}). 
 
\subsection*{Isotrivial $K3$ surfaces} 
A classical example of an isotrivial $K3$ surface is the Kummer surface $X=\Kum(A)$ associated to an abelian surface $A$. It is now known that the class of isotrivial $K3$'s goes much beyond the Kummer surfaces.

We recall that an automorphism $g$ of a $K3$ surface $X$ is called \textit{non-symplectic}, if it acts nontrivially on the nowhere vanishing holomorphic $2$-form $\omega$. In fact, if $g$ has order $p$, a prime number, then $g(\omega)=\zeta_p\omega$, where $\zeta_p$ is a primitive $p$-th root of unity. In \cite{GP} Garbagnati and Penegini proved that most $K3$ surfaces admitting a non-symplectic automorphism of order $p$ can be realized as the mimimal resolution of singularities of a product $C_1\times C_2$ by a diagonal action of $\Z/p\Z$ or $\Z/2p\Z$\footnote{In fact a maximal irreducible component of the moduli space of $K3$ surfaces with a nonsymplectic automorphism has such a structure.}. Moreover, they show that the curve $C_2$ can be taken to be the hyperelliptic curve  $D_p: y^2=u^p-1$. This curve has an automorphism $\delta_p:(x,y)\mapsto(\zeta_p x, y)$  of order $p$, and an automorphism $\tau_p$ of order $2p$ given by the composition of $\delta_p$ with the hyperelliptic involution. More precisely they prove the following theorem. 
\begin{theo} (\cite[Theorem 1.1]{GP})
    Let $S$ be a $K3$ surface admitting a non-symplectic automorphism of order $p=3$ (resp. $p=5, 7, 11, 13$, $p=17, 19$) whose fixed locus contains at least $2$ (resp. $0, 1$) curves. Then $S$ is the minimal resolution of singularities of the quotient $(C_1\times D_p)/(g_1\times \tau_p)$, where $g_1$ is an automorphism of $C_1$ of order $2p$. The non-symplectic automorphism of order $p$ on $S$ is induced by the automorphism $id\times \delta_p$. 
\end{theo}

Using this recipe, they gave explicit examples (including equations) of $K3$ surfaces for each one of the primes $p=3, 5, 7, 11$ (see \cite[Section 7]{GP}). We note that the surfaces corresponding to $p=3$ are elliptic, but most of the others are not.

\subsubsection{Symmetric Squares and geometrically simple Abelian Surfaces}\label{sec:Jac}
The simplest example of a product-quotient surface of mixed type is the symmetric square $\Sym^2(C)=(C\times C)/\langle\sigma\rangle$ of a smooth projective curve $C$ over $k$, where $\sigma$ is the automorphism permuting the two factors. Note that since the fixed locus of $\sigma$ is $1$-dimensional, the surface $\Sym^2(C)$ is smooth. We obtain the following Corollary. 

\begin{cor}\label{cor:Sym2}
    Let $C$ be a smooth projective connected curve of genus $2$ over $k$. Suppose that the Jacobian $J_C$ of $C$ has a mixture of potentially good ordinary and multiplicative reduction. Then \autoref{mainconj} is true for $\Sym^2(C)$ and $J_C$. 
\end{cor}
\begin{proof} We may assume that $C(k)\neq\emptyset$. 
Consider the composition, 
\[\sigma: C\times C\xrightarrow{\pi}\Sym^2(C)\xrightarrow[\sim]{f} \Pic^2(C)\simeq \Pic^0(C)=J_C.\] Here $\pi$ is the degree $2$ quotient map, $f$ is the map that sends a divisor of degree $2$ to its class in $\Pic^2(C)$, and the isomorphism $\Pic^2(C)\simeq\Pic^0(C)$ is defined using a basepoint $P_0\in C(k)$. The map $f$ is birational (\cite[III.5, Theorem 5.1]{Milne}). Thus, this composition is a generically finite morphism of degree $2$ and the result follows from \autoref{thm:main1}. 
    
\end{proof}

In \cite{Gaz19} the first author showed that for an abelian variety $A$ over $k$ with good ordinary reduction, the group $F^2(A)/F^2(A)_{\dv}$ is torsion, but no more information was given. The above Corollary allows us to  improve this result for abelian surfaces. 
\begin{cor}\label{cor:absurface} Let $A$ be an abelian surface over $k$ with a mixture of potentially good ordinary and multiplicative reduction. Then \autoref{mainconj} is true for $A$ and for the Kummer $K3$ surface $\Kum(A)$ associated to $A$. 
\end{cor}
\begin{proof} Since we have a generically finite rational map $\pi: A\dashrightarrow \Kum(A)$, the claim for the Kummer surface will follow from the claim for $A$. 
Since we are allowed to extend to a finite extension, we may assume that $A$ is isogenous to a principally polarized abelian surface $B$. 
The case when $B=E_1\times E_2$ is a product of elliptic curves is already proved in \cite[Theorem 1.2]{GL24}. The case when $B=J_C$ follows from \autoref{cor:Sym2}. The only thing we need to verify is that isogenous abelian surfaces have the same reduction type. 

It follows easily from the N\'{e}ron-Ogg-Shafarevich criterion that  $A$ has good reduction if and only if the same is true for $B$ (\cite[VII, Corollary 7.2]{Sil}). Suppose first that both surfaces have good reduction. In this case if $\phi:A\to B$ is an isogeny, then $\phi$ induces an isogeny $\overline{\phi}: \overline{A}\to\overline{B}$ between the reductions of $A, B$. Since $\ker(\overline{\phi})$ is finite, $B$ must have the same $p^r$-torsion as $A$ for every $r\geq 1$. 

Next assume that $A, B$ have bad reduction. By extending the base field if necessary, we may assume that $A, B$ have a mixture of good ordinary and split multiplicative reduction (see \autoref{reddef}). Then the connected component $\mathcal{A}_s^0$ of the zero element of the special fiber of the N\'{e}ron model $\mathcal{A}$ of $A$ fits into a short exact sequence 
\[0\to T\to \mathcal{A}_s^0\to X\to 0,\] where $T$ is a torus and $X$ is an ordinary abelian variety over the finite field $\kappa$ and the same holds for $\mathcal{B}_s^0$. The reduction morphism $\overline{\phi}: \mathcal{A}_s^0\to \mathcal{B}_s^0$ must preserve the toric pieces and since $\phi$ is a finite surjective map, the dimensions of the tori and the abelian quotients must be the same. To show that both abelian quotients are ordinary one can argue similarly to the good reduction case.

\end{proof}

\subsubsection{Fermat Surfaces}\label{sec:Fermat} We close this subsection by considering the diagonal hypersurfaces of degree $m$ in $\mathbb{P}_k^3$
\[X_m=\{a_0x_0^m+a_1x_1^m+a_2x_2^m+a_3x_3^m=0\}\subset\mathbb{P}_k^3,\] where $m\geq 4$, and $a_0,a_1,a_2, a_3\in k^\times$. 
Over the extension $L=k(\mu_m,\sqrt[m]{a_0},\ldots, \sqrt[m]{a_3})$ the surface $X_m$ becomes isomorphic to the Fermat surface 
\[F_m=\{x_0^m+x_1^m+x_2^m+x_3^m=0\}\subset\mathbb{P}_k^3.\]
Moreover, over the extension $L(\mu_{2m})$, there is a dominant rational map 
\[\phi: C_m\times C_m\dashrightarrow F_m,\]
constructed by Shioda (\cite[(2.1)]{Shi}), where $C_m$ is the Fermat curve of degree $m$, $C_m=\{x_0^m+x_1^m+x_2^m=0\}\subset\mathbb{P}_k^2$. The map $\phi$ is defined as follows,
\[((x_0:x_1:x_2), (y_0:y_1,y_2))\mapsto(x_0y_2:x_1y_2:\zeta_{2m}x_2y_0:\zeta_{2m}x_2y_1),\] 
where $\zeta_{2m}$ is a primitive $2m$-th root of unity. 
When $m=4$, $X_m$ is a $K3$ surface, which is geometrically isomorphic to a Kummer surface $\Kum(A)$ associated to an abelian surface $A$ isogenous to a self-product $E\times E$ of a CM elliptic curve (see \cite[Subsection 2.4]{GL24}). For every $m\geq 5$, $X_m$ is a surface of general type. 

The structure of the Jacobian variety $J_{\Tilde{C}_m}$ of the isomorphic Fermat curve $\Tilde{C}_m=\{x_0^m+x_1^m=x_2^m\}\subset\mathbb{P}_k^2$ over a finite field is well-known (\cite{Yui}). In fact it follows that for every prime $p\equiv 1\mod m$, $J_{\Tilde{C}_m}$ has good ordinary reduction (\cite[Proposition 4.1, Theorem 4.2]{Yui}). Thus, we obtain the following Corollary.

\begin{cor}
   Let $X_m$ be a diagonal hypersurface of degree $m$ in $\mathbb{P}_k^3$ where $k$ is a $p$-adic field with  $p\equiv 1\mod m$. Then \autoref{mainconj} is true for $X_m$.  
\end{cor}

\subsection{Some computations of Suslin's Homology}\label{sec:Suscomputs} We close this article by giving some examples illustrating that our methods have the potential to give explicit information on Suslin's homology. 

Let $X$ be a smooth projective surface over a $p$-adic field $k$ and $j:U\hookrightarrow X$ a dense open immersion. 
As we saw in the proof of \autoref{thm:main2}, when $X-U$ is $0$-dimensional, we have an equality $F^2(U)=F^2(X)$.
On the other hand, if $X-U$ is a divisor, it is very often the case that $F^2(U)$ is much larger than $F^2(X)$. 
This can already be seen in the following classical examples. 

\begin{exm}
    Let $X=\mathbb{P}^1_k\times \mathbb{P}^1_k$ and $U=\G_m\times \G_m$. Since $X$ is a rational surface, $F^2(X)=0$. On the other hand, there is an isomorphism $F^2(U)\simeq K_2^M(k)$. The $K$-group $K_2^M(k)$ is the direct sum of a finite group and an uncountable uniquely divisible group.
     \end{exm}

    \begin{exm}\label{ruled-surface example}

    Consider $X=C\times \mathbb{P}^1_k$, where $C$ is a smooth projective connected curve over $k$ with $C(k) \neq \emptyset$. Since $X$ is a ruled surface, we again have $F^2(X)=0$. Consider the open subvariety $U=C\times \G_m$. Then there is an isomorphism $F^2(U)\cong V(C)$ (see \cite[1.7]{KY} and \cite[2.4.4]{RS}), which is the direct sum of a finite group and a divisible group (\autoref{CFTcurves}).
    Assume further that genus of curve, $g(C) \geq 1$, then \( V(C) \otimes_{\mathbb{Z}} \mathbb{Q} \) is an infinite-dimensional \( \mathbb{Q} \)-vector space (\cite[Page~210]{Sz}). Furthermore, if \( C \) is a Mumford curve, then \( V(C)_{\mathrm{tor}} \) is infinite (see \cite[Page 178]{AS}). Hence $V(C)_{\tor}$, the torsion subgroup and $V(C)/V(C)_{\tor}$, the uniquely divisible subgroup of $V(C)$ (= $F^2(U)$) are both infinite.

\end{exm}  

 One thing that these two examples have in common is that their Albanese variety $Alb_U$ has nontrivial toric piece. Generalizing these examples, B. Kahn and T. Yamazaki \cite{KY}  gave a $K$-theoretic description of $F^2(U)$ for 
$U=C_1\times C_2$ a product of two smooth quasi-projective curves over $k$, generalizing the work of Raskind and Spiess (\cite{RS}) on products of projective curves. Let $\overline{C_i}$ be the smooth compactification of $C_i$ and $X=\overline{C_1}\times\overline{C_2}$ so that we have an open immersion $j:U\hookrightarrow X$. Suppose that $\overline{C_i}-C_i$ is the support of a reduced divisor $\mathfrak{m}_i$. They showed   an isomorphism
 \begin{equation}\label{opencurves}
     F^2(U)\simeq K(k;J_{\mathfrak{m}_1}, J_{\mathfrak{m}_2})
 \end{equation} (\cite[1.7]{KY}), 
    where $J_{\mathfrak{m}_i}$ is the generalized Jacobian of $\overline{C_i}$ corresponding to the modulus $\mathfrak{m}_i$ (see section \ref{sec:genJac}) and $K(k;J_{\mathfrak{m}_1}, J_{\mathfrak{m}_2})$ is the Somekawa $K$-group attached to the semi-abelian varieties $J_{\mathfrak{m}_1}, J_{\mathfrak{m}_2}$ (see \cite{Som}).
    We can deduce from this that an equality $F^2(X)=F^2(U)$ holds if and only if $J_{\mathfrak{m}_i}$ coincides with the usual Jacobian of $\overline{C_i}$ for $i=1,2$, if and only if $Alb_X=Alb_U$. Recalling the dimension of the maximal subtorus of the generalized Albanese variety from \autoref{genalbsection} we see that this happens if and only if the support of each $\mathfrak{m}_i$ is at most one rational point. 

    One can thus pose the question: how large can $F^2(U)$ be when $Alb_X=Alb_U$? In what follows we focus on the simplest case that this happens, namely when $U=X-C$ for some projective smooth connected curve $C\hookrightarrow X$ (see Example \ref{onecomponent}). We show that even in this case the structure of $F^2(U)$ can really vary.   

\begin{exm}\label{P2 example}
 Let \( C \) be a smooth projective connected curve contained in \( X = \mathbb{P}^2_k \) with \( C(k) \neq \emptyset \), and let \( U = \mathbb{P}^2_k - C \). Then we have \( F^{1}(X) = F^{2}(X) = 0 \). Moreover, since \( Alb_{U}=Alb_X=0 \) (see Example \ref{onecomponent}), we have \( F^{1}(U) = F^{2}(U) \). Therefore,
\[
F^{2}(U) \cong \operatorname{coker} \left( \CH^{2}(X,1) \xrightarrow{g} SK_{1}(C) \right)
\]
(see~\autoref{lem:maindiagram}). The group \( SK_{1}(C) \cong k^{\times} \oplus V(C) \). By the projective bundle formula (see \cite{Bl2}), we have $\CH^{2}(\mathbb{P}^2_k, 1) \cong \CH^{0}(k, 1) \oplus \CH^{1}(k, 1) \oplus \CH^{2}(k, 1)$. The group $ \CH^{2}(k, 1)$ vanishes by a dimension argument, and \( \CH^{0}(k, 1) \) vanishes by \cite[Corollary~4.2]{MVW}. Hence,
$\CH^{2}(\mathbb{P}^2_k, 1) \cong \CH^{1}(k, 1) \cong k^{\times}$. Also, the composite map
\[
\CH^{2}(X,1) \xrightarrow{g} SK_{1}(C) \xrightarrow{\pi_{2}} k^{\times}
\]
has finite cokernel (see Proposition~\ref{prop:secsoffibration}). Therefore, we have $F^2(U) \cong V(C)$ (modulo finite kernel and cokernel).
In particular, if \(V(C)_{\mathrm{tor}}\) is infinite, then so is \(F^2(U)_{\mathrm{tor}}\). 
Furthermore, if the genus of \( C \) i.e., $g(C) \geq  1$, then \( V(C) \otimes_{\mathbb{Z}} \mathbb{Q} \) is an infinite-dimensional \( \mathbb{Q} \)-vector space (\cite[Page~210]{Sz}). It then follows from the above discussion that \( F^{2}(U) \otimes_{\mathbb{Z}} \mathbb{Q} \) is also an infinite-dimensional \( \mathbb{Q} \)-vector space. \\ %This shows that \( F^2(U) \) can be ``too large'' even for curves $C$ embedded inside projective space $\mathbb{P}^{2}_k$.

\end{exm}
\begin{rem}
 This example can be further extended to surfaces which are $\mathbb{P}^1$-bundle.  These examples show that \(F^2(U)\) can be too large. In fact, they suggest that the torsion divisible subgroup \((F^2(U)_{\mathrm{div}})_{\mathrm{tor}}\), the uniquely divisible quotient $F^2(U)_{\mathrm{div}}/(F^2(U)_{\mathrm{div}})_{\mathrm{tor}}$
 and the finite summand of \(F^2(U)\) can all be arbitrarily large, even when \(F^2(X) = 0\).
\end{rem}

 On the other extreme, generalizing the example of Kahn-Yamazaki, the next proposition shows that when the embedding $C\hookrightarrow X$ is ``special", then an equality $F^2(X)=F^2(U)$ can happen.

\begin{prop}\label{prop:secsoffibration} Let $\pi:X\to C$ be a smooth  fibration, where $C$ is a smooth projective curve over $k$ with $C(k)\neq\emptyset$. Let $s:C\hookrightarrow X$ be a section of the fibration and  $j:U\hookrightarrow X$ be the complement of the section. Then $F^2(U)=F^2(X)$. 
\end{prop}

\begin{proof} Let $J_C$ be the Jacobian of $C$, and $P_0\in C(k)$ a fixed rational point. 
Keeping the same notation as in the proof of \autoref{thm:main2}, we have an exact sequence
\[\CH^2(X,1)\xrightarrow{g}SK_1(D)\xrightarrow{f}F^1(U)\xrightarrow{\beta}F^1(X)\to 0.\] Since $D=s(C)\cong C$ is a smooth curve, it follows that $Alb_U=Alb_X$ (see Example \ref{onecomponent}). Thus, the conclusion of the proposition will follow from the following.

\textbf{Claim:} The map $\CH^2(X,1)\xrightarrow{g}SK_1(s(C))$ is surjective.

As in the previous example, we have a decomposition $SK_1(s(C))\cong SK_1(C)\cong V(C)\oplus k^\times$, where $V(C)=\ker(SK_1(C)\xrightarrow{N}k^\times)$ is the kernel of the norm map. Note first that the pullback map $s^\star:\NS(X)\to\NS(s(C))\simeq\Z$ is surjective. To see this, consider the fiber $X_{P_0}$ above $P_0$. Then 
\[s^\star([X_{P_0}])=[X_{P_0}]\cdot [s(C)]=1.\]

It remains to show that $V(C)$ is contained in the image of $g$.  It is well-known (\cite[Theorem 2.1]{Som}) that there is an isomorphism $V(C)\simeq K(k;J_C,\G_m)$, with the Somekawa $K$-group attached to $J_C$ of $C$ and $\G_m$. This group admits a surjection 
\[\bigoplus_{L/k\text{ finite}}J_C(L)\times L^\times\twoheadrightarrow K(k;J_C,\G_m).\] In particular, the group $SK_1(C)\simeq \CH^2(C, 1)$ is generated by decomposable $(2,1)$-cycles. We have a commutative square,
\[\begin{tikzcd} \bigoplus_{L/k} \Pic^0(X_L)\otimes L^\times \ar{r} \ar{d}{s^\star} & \CH^2(X,1) \ar{d}{g} \\
\bigoplus_{L/k} J_C(L)\otimes L^\times \ar{r} & SK_1(s(C)) 
\end{tikzcd}.\]  Thus, it suffices to show that for every finite extension $L/k$ the pullback map $s^\star:\Pic^0(X_L)\to \Pic^0(C_L)= J_C(L)$ is surjective. 
The morphisms $\pi, s$ induce homomorphisms $\pi_\star:Alb_X\to Alb_C=J_C$, and $s_\star: Alb_C\to Alb_X$ such that $\pi_\star\circ  s_\star=1_{J_C}$. The dual homomorphisms,
$\widehat{\pi_\star}: \Pic^0(C)=J_C\to \Pic^0(X)$, $\widehat{s_\star}:\Pic^0(X)\to\Pic^0(C)$ thus satisfy $\widehat{s_\star}\circ\widehat{\pi_\star}=1_{J_C}$. It follows that the map $\widehat{s_\star}:\Pic^0(X)\to\Pic^0(C)$ is surjective, which is none other than the pullback $s^\star$. 

\end{proof}

\begin{exm} Generalizing the result of Kahn and Yamazaki, let $X=C\times C$ be the self-product of a smooth projective curve over $k$ with $C(k)\neq\emptyset$. Consider the diagonal embedding $\Delta:C\hookrightarrow X$ and let $U=X-\Delta$. Then $F^2(U)=F^2(X)$. 
\end{exm}

\begin{exm}\label{ex:bielliptic} Let $A=E_1\times E_2$ be a product of elliptic curves over $k$. Let $G$ be a finite group acting on $E_1$ by translations by a fixed torsion point $P_0\in E_1(k)$ and on $E_2$ by automorphisms such that $E_2/G\cong\mathbb{P}^1_k$. Consider the diagonal action of $G$ on $A$. This is a fixed point free action, and hence the quotient surface $X=A/G$ is smooth. Such surfaces are called \textit{bielliptic}, or \textit{hyperelliptic}. The Albanese variety of $X$ is $Alb_X=E_1'$, where $E_1'$ is the elliptic curve $E_1/\langle P_0\rangle$, and the Albanese map $\alb_X: X\to E_1'$ is an elliptic fibration. The surface $X$ has Kodaira dimension $0$ and geometric genus $p_g=0$. Thus, it follows by \cite[Th\'{e}or\`{e}me D]{CTR} and the main result of \cite{BKL}   that $F^2(X)$ is finite.  It follows from \autoref{prop:secsoffibration} that for every section $s: E_1'\hookrightarrow X$  of the fibration,  $F^2(X-s(E_1'))=F^2(X)$. In particular, it is finite. 
\end{exm}

\begin{rem}
    The above examples suggest that the structure of $F^2(U)$ in the case when $U = X - C$ depends on $V(C)$, the cycle-theoretic invariant of $C$. At the same time, it also depends on the embedding of the curve, as shown in \autoref{prop:secsoffibration}. 
\end{rem}

\end{document}